\documentclass{amsart}
\usepackage{amsmath}
\usepackage{amsthm}
\usepackage{amssymb}
\usepackage{tikz}
\usetikzlibrary{automata,positioning}
\usepackage{enumitem}
\usepackage{mathrsfs}
\usepackage{scalerel}

\usepackage{hyperref}

\newtheorem{theorem}{Theorem}[section]
\newtheorem{proposition}[theorem]{Proposition}

\newtheorem{alphatheorem}{Theorem}

\newtheorem*{convention*}{Convention}

\newtheorem{corollary}[theorem]{Corollary}

\theoremstyle{definition}

\newtheorem*{proposition*}{Proposition}
\newtheorem*{observation*}{Observation}
\newtheorem*{claim*}{Claim}

\newtheorem*{lemma*}{Lemma}

\newtheorem{remark}[theorem]{Remark}
\newtheorem*{remark*}{Remark}
\newtheorem{conjecture}[theorem]{Conjecture}
\newtheorem*{conjecture*}{Conjecture}

\renewcommand{\Re}{\operatorname{Re}}
\theoremstyle{plain}
\newtheorem{lemma}[theorem]{Lemma}

\newcommand{\bra}[1]{ \left( #1 \right) }

\renewcommand{\tilde}{\widetilde}

\newcommand{\abs}[1]{\left|#1\right|}

\newcommand{\norm}[1]{\left\lVert #1 \right\rVert}

\def\LMS{0}
\def\true{1}

\newcommand{\ifLMS}[2]{
\if\LMS\true
#1
\else
#2
\fi
}

\newcommand{\cP}{\mathcal{P}}

\newcommand{\fpa}[1]{\left\lVert #1 \right\rVert}\newcommand{\e}{\varepsilon}
\renewcommand{\a}{\alpha}

\usepackage{mathrsfs}
\usepackage[mathscr]{euscript}

\newcommand{\NN}{\mathbb{N}}
\newcommand{\QQ}{\mathbb{Q}}

\newcommand{\OO}{\tilde{O}}
\newcommand{\PP}{\mathbb{P}}

\DeclareMathOperator*{\EEE}{\scalerel*{\mathbb{E}}{\textstyle\sum}}
\DeclareMathOperator*{\PPP}{\scalerel*{\mathbb{P}}{\textstyle\sum}}

\newcommand{\ZZ}{\mathbb{Z}}
\newcommand{\RR}{\mathbb{R}}

\newcommand{\CC}{\mathbb{C}}
\newcommand{\UU}{\mathbb{S}^1}
\newcommand{\EE}{\mathbb{E}}

\newcommand{\cF}{\mathcal{F}}

\newcommand{\floor}[1]{\left\lfloor #1 \right\rfloor}

\newcommand{\set}[2]{\left\{ #1 \ \middle| \ #2 \right\} }

\newcommand{\parbreak}[1]{
\begin{center}
***
\end{center}
}

\usepackage{color}

\newcommand{\bb}{\mathbf}

\newcommand{\conj}{\mathscr{C}}

\usepackage{soul}

\usepackage{stmaryrd}
\newcommand{\ifbra}[1]{\left\llbracket #1 \right\rrbracket}


\begin{document}

\author[A.\ Fan]{Aihua Fan}
\address[A.\ Fan]{LAMFA, UMR 7352 CNRS, University of Picardie, 33 Rue Saint Leu, 80039, Amiens, France}
\email{ai-hua.fan@u-picardie.fr}

\author[J.\ Konieczny]{Jakub Konieczny}
\address[J.\ Konieczny]{Einstein Institute of Mathematics Edmond J. Safra Campus, The Hebrew University of Jerusalem Givat Ram. Jerusalem, 9190401, Israel}
\address{Faculty of Mathematics and Computer Science, Jagiellonian University in Krak\'{o}w, \L{}ojasiewicza 6, 30-348 Krak\'{o}w, Poland}
\email{jakub.konieczny@gmail.com}

\title{On uniformity of $q$-multiplicative sequences}

\maketitle 

\begin{abstract}
	We show that any $q$-multiplicative sequence which is \emph{oscillating} of order $1$, i.e.\ does not correlate with linear phase functions $e^{2\pi i n \alpha}$ ($\alpha \in \RR)$, is Gowers uniform of all orders, and hence in particular does not correlate with polynomial phase functions $e^{2\pi i p(n)}$ ($p \in \RR[x]$). Quantitatively, we show that any $q$-multiplicative sequence which is of Gelfond type of order 1 is automatically of Gelfond type of all orders. Consequently, any such $q$-multiplicative sequence is a good weight for ergodic theorems. We also obtain combinatorial corollaries concerning linear patterns in sets which are described in terms of sums of digits. 
\end{abstract}

\makeatletter{}\section{Introduction}

    Let $q\ge 2$ be an integer. A sequence $f \colon \NN_0 \to \UU := \set{ z \in \CC}{\abs{z} = 1}$ is said to be \emph{$q$-multiplicative} if for any integer $t \in \NN_0$ and any integers $m, n \in \NN_0$ with $m < q^t$ and $q^t \mid n$ we have $f(m+n) = f(m) f(n)$
	\cite{Gelfond-1968}. A typical example of a $q$-multiplicative sequence is $\big(e( \alpha s_q(n))\big)_{n\ge 0}$
	where $e(t) = e^{2\pi i t}$, $\alpha$ is a real parameter and $s_q(n)$ is the sum of digits of 
	$n$ in its $q$-adic expansion. The classical Thue--Morse sequence
	corresponds to $\alpha=\frac{1}{2}$ and $q=2$. Combinatorial and number theoretic properties of such sequences have been extensively studied, both in the case of the sum-of-digits function \cite{MendesFrance-1967}, \cite{FouvryMauduit-1996-AA}, \cite{FouvryMauduit-1996-MA}, \cite{MauduitSarkozy-1996}, \cite{MauduitSarkozy-1997}, \cite{Drmota-2001}, \cite{FouvryMauduit-2005},  \cite{MauduitPomeranceSarkozy-2005}, \cite{DartygeTenenbaum-2006}, \cite{Hofer-2007} and in the case of general $q$-multiplicative sequences \cite{Gelfond-1968}, \cite{Delange-1972}, \cite{Coquet-1975}, \cite{Coquet-1976}, \cite{Coquet-1977}, \cite{Coquet-1980}, \cite{LesigneMauduit-1996}, \cite{IndlekoferKatai-2002}, \cite{IndlekoferKataiLee-2002}, \cite{Katai-2002}, \cite{IndlekoferLeeWagner-2005}, \cite{Mauclaire-2005}, \cite{MauduitRivat-2005}, \cite{Katai-2009}.

Note that if $f$ is $q$-multiplicative then necessarily $f(0) = 1$ and  the equality $f(m+n) = f(m) f(n)$ holds also under the weaker assumption that the sets of positions where non-zero digits appear in expansions of $m$ and $n$ in base $q$ are disjoint. Hence, a $q$-multiplicative sequence $f$ is uniquely described by its values $f(a q^t)$ for $0 \leq a < q$ and $t \geq 0$, and
\[
	f\bra{ n } = \prod_{t=0}^\infty f\bra{ n^{(t)} q^t} \quad \text{where } n = \sum_{t=0}^\infty n^{(t)}q^t, \ 0 \leq n^{(t)} < q.
\]

Ces\`{a}ro convergence of $q$-multiplicative sequences was studied by Delange \cite{Delange-1972}. As a straightforward consequence of results in \cite[p.\ 292,\ Th\'eor\`eme 1]{Delange-1972}, for any $q$-multiplicative sequence $f \colon \NN_0 \to \UU$ we have
\begin{equation}\label{eq:95:00}
	\EEE_{n<N} f(n) = \prod_{t\leq {\log_q} N} \EEE_{a < q} f(aq^t) + o(1) \text{ as } N \to \infty,
\end{equation}
where $o(1)$ denotes an error term which tends to $0$ as $N \to \infty$. As a consequence (\cite[Th\'eor\`eme 2]{Delange-1972}) the Ces\`{a}ro mean of $f$ is $0$ unless 
\begin{equation}\label{eq:95:09}
\sum_{t=0}^\infty \bra{1 - \Re{ \EEE_{a<q} f(aq^t)}} < \infty.
\end{equation}

In particular, the mean of $f$ is $0$ for (Lebesgue) almost all choices of the values $f(aq^t)$. For many $q$-automatic sequences we have a much stronger estimate
\begin{equation}\label{eq:osc-01a}
\sup_{\alpha \in \RR} \abs{ \EEE_{n<N} f(n)e(n\alpha) } \ll N^{-c}
\end{equation} 
for a constant $c > 0$.
Using the terminology introduced in \cite{Fan2017}, \eqref{eq:osc-01a} means that $f$ is of Gelfond type with Gelfond exponent bounded by $1-c$; this condition can be construed as a weak notion of uniformity. Examples of $q$-multiplicative sequences which are of Gelfond type include the Thue--Morse sequence and more generally all non-periodic strongly $q$-multiplicative sequences (see Proposition \ref{prop:strong-multi-dichotomy} for details).

In this note, we show that as soon as a $q$-multiplicative sequence obeys condition \eqref{eq:osc-01a}, it exhibits uniformity of higher order. Our first result is the  following uniform estimate on trigonometric polynomials having a $q$-multiplicative sequence as coefficients. 
Note that if $f$ is $q$-multiplicative then so is the sequence $n \mapsto f(n)e(\alpha n)$, but this is no longer the case for the sequence $n \mapsto f(n) e(p(n))$ when $p$ has degree larger than $1$.

\begin{alphatheorem}\label{thm:oscillating=>fully-oscillating}
	Let $f \colon \NN_0 \to \UU$ be a $q$-multiplicative sequence of Gelfond type (cf.\ \eqref{eq:osc-01a}). 	Then for any $d \geq 1$ there exists $c_d > 0$ such that 
	\begin{equation}\label{eq:osc-02}
		\sup_{\substack{ p \in \RR[x] \\ \deg p \leq d}} \abs{ \EEE_{n < N} f(n) e(p(n))} \ll N^{-c_d}. 
	\end{equation}
\end{alphatheorem}

In fact, we shall prove a stronger statement, which concerns an estimate for
the Gowers norms of $q$-multiplicative sequences.

\begin{alphatheorem}\label{thm:oscillating=>Gowers-uniform}
For any integer $s \geq 2$ there exists a constant $\kappa_s > 0$ such that the following holds. Let $f \colon \NN_0 \to \UU$ be a $q$-multiplicative sequence. Then 
	\begin{equation}\label{eq:osc-102}
		\norm{f}_{U^s[N]} \ll \norm{f}_{U^2[N]}^{\kappa_s}.
	\end{equation}	
\end{alphatheorem}

	In particular (with the same notation as above) if $f$ is of Gelfond type then $\norm{f}_{U^s[N]} \ll N^{-c_s'}
$ and \eqref{eq:osc-02} holds for some constants $c_s',c_d > 0$. Moreover, if
\begin{equation}\label{eq:139:00}
	\EEE_{n<N} f(n) e(n\a) \to 0 \text{ as } N \to \infty \text{ for any } \a \in \RR.
\end{equation}
then the convergence is automatically uniform in $\alpha$ (see Lemma \ref{lem:U2-uniform}) whence $f$ is Gowers uniform of order $2$ (see Section \ref{sec:Background} for connection between Fourier transform and uniformity norms) and consequently
$$\norm{f}_{U^s[N]} \to 0 \text{ as } N \to \infty \text{ for any } s \geq 2.$$

	Let us briefly outline the proof of Theorem \ref{thm:oscillating=>Gowers-uniform}. Using similar ideas as in \cite{Konieczny-2017+}, we obtain recursive relations between Gowers norms of $f$ of order $s$ (and analogous expressions) on intervals of length $q^L$, $L \geq 0$ (Lemma \ref{lem:recurrence}, this can be compared with \eqref{eq:95:00}). A new phenomenon we have to deal with is that it is possible that at some steps of the recurrence the norm does not decrease. The key idea for dealing with this lack of cancellation is to apply an inverse theorem for Gowers uniformity norms to extract some polynomial behaviour of a suitable restriction of $f$. In fact, since $f$ is $q$-multiplicative, the only polynomials that it can resemble are the linear ones (Proposition \ref{prop:rigidity}). Ultimately, we are able to conclude from this that if at a given step in the recurrence the norm of order $s$ does not decrease significantly then nor does the norm of order $2$. Consequently, the norms of order $s$ and of order $2$ have comparable order of magnitude (Proposition \ref{prop:oscillating=>E-saving}). A potentially interesting feature of our argument is that we are able to avoid the powerful inverse theorem of Green, Tao and Ziegler \cite{GreenTaoZiegler-2012} and instead use the much simpler (both in proof and in formulation) $99\%$ variant.

As a consequence of Theorem \ref{thm:oscillating=>fully-oscillating}
and Corollary 2 in \cite{Fan2017c},
we get immediately the following weighted ergodic theorem, which states that Gelfond type
$q$-multiplicative sequences are good weights for polynomial ergodic theorem. An analogous result for more general $q$-multiplicative sequences and (linear) ergodic theorem was obtained in \cite{LesigneMauduitMosse-1994}, see also \cite{LesigneMauduit-1996}. Similar results for automatic sequences were obtained in \cite{EisnerKonieczny-2018}. For more background on weighted ergodic theorems and associated bibliography, we refer to \cite{Fan2017c} and \cite{EisnerKonieczny-2018}.

\begin{alphatheorem}\label{thm:WET}
Suppose that  $f \colon \NN_0 \to \UU$ is a $q$-multiplicative of Gelfond type (cf. \eqref{eq:osc-01a}). Then for any real polynomial $p$ with $p(\NN_0) \subset \NN_0$, for any measure-preserving dynamical system
$(X, \mathcal{B}, T, \mu)$ and any function $F \in L^r(\mu)$ (for some $r>1$) we have
\begin{equation}\label{eq:WET}
\lim_{N\to \infty} \EEE_{n <N} f(n) F(T^{p(n)} x) =0 \quad \text{ for $\mu$-a.a.\ } x \in X.
\end{equation}
\end{alphatheorem}

Using standard tools of higher order Fourier analysis we can derive several combinatorial corollaries of Theorem \ref{thm:oscillating=>Gowers-uniform}. In \cite[Cor.\ 2.3]{MullnerSpiegelhofer-2017}, the authors show that any finite sequence of $\pm 1$'s appears in the Thue--Morse sequence along an arithmetic progression. A more detailed counting statement for the Thue--Morse sequence is obtained in \cite{Konieczny-2017+}. Here we obtain a further generalisation.

\begin{alphatheorem}\label{thm:sum-of-digits-patterns}
	Let $\alpha \in \RR \setminus \QQ$ and $q \in \NN_{\geq 2}$. Then for any $k \in \NN$ and any (non-degenerate) intervals $I_0, \dots, I_{k-1} \subset [0,1)$ there exists an arithmetic progression $n,n+m,\dots,n+(k-1)m$ of length $k$ such that $\alpha s_q(n + j m) \bmod{1} \in I_j$ for all $0 \leq j < k$. 
	
	Similarly, let $Q \in \NN$ be coprime to $q-1$. Then for any $k$ and any $r_0,\dots, r_{k-1} \in [Q]$ there exists an arithmetic progression $n,n+m,\dots,n+(k-1)m$ of length $k$ such that $s_q(n + j m) \equiv r_j \bmod{Q}$ for all $0 \leq j < k$. 
	
    Moreover, in both of the above statements the number of such progressions contained in $[N]$ grows proportionally to $N^2$ as $N \to \infty$ and all other parameters are fixed.
\end{alphatheorem}

While the derivation of the above statement follows from Theorem \ref{thm:oscillating=>Gowers-uniform} by standard methods, for the sake of completeness we include the proof in Section \ref{sec:Epilogue}.

We also point out that there are not many explicit examples of sequences known to be Gowers uniform of all orders. These examples include the Thue--Morse and the Rudin--Shapiro sequences, see \cite{Konieczny-2017+} for extended discussion. Hence, Theorem \ref{thm:oscillating=>Gowers-uniform} can be of interest as a new source of such examples.

\subsection*{Notation}
We use the standard asymptotic notation. For two quantities $X$ and $Y$ (possibly dependent on some parameters) we write $X = O(Y)$
or $X \ll Y$ if there exists an absolute constant $c > 0$ such that $\abs{X} \leq c Y$. If $c$ is allowed to depend on a quantity $Z$ then we write $X = O_Z(Y)$ or $X\ll_Z Y$. If additionally $\abs{c} \leq 1$ we write $X = \OO(Y)$.

We write $\NN$ for the set of positive integers $\{1,2,3,\dots\}$, $\NN_{\geq s} := \set{n \in \NN}{ n \geq s}$ and $\NN_0 := \NN \cup \{0\}$. We shall also denote by $\UU := \set{ z \in \CC}{ \abs{ z } = 1}$ the set of complex numbers of unit modulus.  We use the shorthand $[N] := \{0,1,\dots,N-1\}$ (which is slightly non-standard), $e(x) := e^{2 \pi i x}$ and $\norm{x} : = \norm{x}_{\RR/\ZZ} := \min_{n \in \ZZ} \abs{ n - x}$.

For a finite set $X$, we shall say that a statement $\Phi(x)$ is true for $\e$-almost every $x \in X$ if  the set of $x$ such that $\Phi(x)$ fails has size $\leq \e \abs{X}$. Likewise, we shall say that a statement $\Phi(n)$ is true for $\e$-almost every $n \in \NN_0$ if the set of $n$ for which $\Phi(n)$ fails has upper asymptotic density $\leq \e$. Here, the upper asymptotic density of a set $A \subset \NN_0$ is given by $\bar d(A) := \limsup_{N} \abs{ A \cap [N]}/N$ (lower asymptotic density is defined similarly, with $\liminf$ in place of $\limsup$).

For a function $f$ defined on a finite set $X$, we write $\mathbb{E}_{x \in X} f(x)$ or simply $\mathbb{E}_X f$ for $\frac{1}{|X|}\sum_{x \in X} f(x)$. Accordingly, for a condition $\Phi(x)$ which holds true for at least one $x \in X$, we write $\EE_{x \in X} \bra{ f(x) \middle| \Phi(x)}$ for the expression $\EE_{x \in X} (f(x) 1_{\Phi}(x)) / \PP_{x \in X}(\Phi(x))$ where $1_{\Phi}(x) = 1$ if $\Phi(x)$ is true and $1_\Phi(x) = 0$ if $\Phi(x)$ is false and of course $\PP_{x \in X} (\Phi(x)) = \EE_{x \in X} (1_\Phi(x))$.
In particular, if $\emptyset \neq Y \subset X$ then $\EE_{x \in X} \bra{f(x) \middle| x \in Y} = \EE_{x \in Y} f(x)$. 

\subsection*{Acknowledgements} The authors are grateful to Christian Mauduit for providing an inspiration for this line of research and to Tamar Ziegler and Jakub Byszewski for helpful discussions. The first author is supported by NSFC 11471132 and the second author is supported by ERC grant ErgComNum 682150. 

\makeatletter{}\section{Background}\label{sec:Background}

\subsection*{Oscillating sequences}
 A sequence of complex numbers $w =
(w_n)_{n\ge 0} \subset \UU$  is said to be {\em oscillating  of order
	$d$} ($d \ge  1$ being an integer) if
 for any real polynomial $p$ of degree less than or equal to $d$ we have
\begin{equation}\label{osc_d}
\lim_{N\to \infty} \frac{1}{N}\sum_{n=0}^{N-1} w_n e(p(n)) =0.
\end{equation}
A sequence is said to be {\em fully oscillating} if it is oscillating of all orders.
This terminology was introduced in \cite{Fan2016,Fan2017} (the notion of oscillating sequence of order $1$
was earlier introduced in \cite{FanJiang} and we simply spoke of oscillating sequences)
in order to extend the study of the M\"{o}bius--Liouville randomness law to the class of
oscillating sequences. Examples of oscillating sequences were constructed in \cite{Akiyama-Jiang,Elabdalaoui,Fan2016,Fan2017,Fan2017b}. Shi \cite{Shi} proved that 
for Lebesgue almost all $\beta >1$, the sequence $(e(\beta^n))$ 
is a Chowla sequence, which is a strictly stronger property than being oscillating of all orders. Indeed, any Chowla sequence is orthogonal to all topological dynamical systems of zero entropy.  

Following \cite{Fan2017}, we also introduce a property stronger than being an oscillating sequence. We shall say that $w$ is a {\em sequence of Gelfond type of order $d$} if there exists $ \gamma_d<1$ such that
	\begin{equation}\label{eq:Gelfond}
\sup_{\substack{ p \in \RR[x] \\ \deg p \leq d}} \abs{ \sum_{n = 0}^{N-1} w_n e(p(n))} \ll N^{\gamma_d}. 
\end{equation}
Note that \eqref{eq:Gelfond} can only hold for $\gamma_d \geq \frac{1}{2}$. The infimum of all the values of $\gamma_d$ such that \eqref{eq:Gelfond} holds is called the {\em Gelfond exponent of order $d$}. For the Thue--Morse sequence, which is $2$-multiplicative, Gelfond \cite{Gelfond-1968}
proved that the Gelfond exponent of order $1$ is equal to $\frac{\log 3}{\log 4}=0.792481\cdots$, and Gelfond property for all orders follows from results in \cite{Konieczny-2017+}, although the values of the Gelfond exponents remain unknown for orders $2$ and higher.

There are few results about the exact Gelfond exponents
of sequences. However, it is known that many $q$-multiplicative sequences
are of Gelfond type of order $1$. The interest of Theorem \ref{thm:oscillating=>fully-oscillating}  is that for $q$-multiplicative
sequences, the Gelfond property of order $1$ implies automatically 
the Gelfond property of higher orders. As proved in \cite{Fan2017},
once we prove the Gelfond property (\ref{eq:Gelfond}), we immediately get
a weighted ergodic theorem with weights $w$ along polynomials $p$ (see Theorem \ref{thm:WET}). 

\newcommand{\bla}{\tau}
Let us  look at the generalized Thue--Morse sequences
$(t_n^{(\bla)})$ defined by $t_n^{(\bla)} = e(\bla s_2(n))$, $\bla \in [0,1)$,
which are $2$-multiplicative. They define
trigonometric polynomials which can be written as trigonometric products
\begin{equation}\label{eq:TP}
     \left|\sum_{k=0}^{2^n -1} t_k^{(\bla)} e(k \alpha)\right|
      = 2^{n} \prod_{j=0}^{n-1}|\cos \pi (2^j \alpha +\bla)|.
\end{equation}
The dynamics $x \mapsto 2x \mod 1$ on $\mathbb{R}/\mathbb{Z}$ is then 
naturally involved.  The condition \eqref{eq:osc-01a} is satisfied by $(t_n^{(\bla)})$
for any $0<\bla<1$ ($\bla=0$ gives rise to $t_n^{(0)}=1$ and the condition \eqref{eq:osc-01a}
is not satisfied) because 
$$
     \max_x |\cos \pi (x+\bla) \cos \pi (2x +\bla)|<1.
$$
Mauduit, Rivat and S\'ark\"{o}zy \cite{MRS} showed the following better estimate
\begin{equation}\label{eq:MRS}
 \sup_{ \alpha}\prod_{j=0}^{n-1}\abs{\cos \pi (2^j \alpha +\bla)} \le D_{\bla} \exp\left(-\frac{\pi^2}{20}\fpa{\bla} n\right)
\end{equation}
where $D_{\bla}$ is a constant and $\|\bla\|$ denotes the distance from $c$ to the nearest integer. However, this estimate is not optimal. Recently it is proved in \cite{FSS2018} that the Gelfond exponent of $(t_n^{(\bla)})$ is equal to $1+ \log_2 \beta(\bla)$ where
$$
   \beta(\bla) = \sup_{ \mu } \int \log |\cos\pi (x+\bla)|d\mu(x)
$$
($\mu$ varying over all invariant measure under $x \mapsto 2x \bmod 1$)
and that the supremum is attained at a (unique) Sturmian measure. A method
is also provided in \cite{FSS2018} to compute $\beta(\bla)$ for a large set of $\bla$'s. For example,
 we get the exact value
\begin{equation}\label{eq: beta_2}
\beta(\bla) =\frac{1}{2}\log \left|\cos \pi \left(\frac{1}{3}+\bla\right)
\cos \pi \left(\frac{2}{3}+\bla\right)\right|
\end{equation}
for $\bla \in (0.43,0.57)$.
Similar formulas hold on other intervals of $\bla$ (see \cite{FSS2018} for details).

We point out that arithmetic 
problems involving the sum of dyadic digits depend on the estimates
of $L^p$-norms of the trigometric polynomials in (\ref{eq:TP}), see \cite{Gelfond-1968,MRS,FM1996a}.

\subsection*{Gowers uniformity norms}

The uniformity norms $\norm{\cdot}_{U^s[N]}$ were first introduced by T.\ Gowers in his work on a (higher order) Fourier-analytic proof of Szemer\'{e}di's Theorem \cite{Gowers-2001}. Here we just give the basic facts; for an extensive background we refer  to \cite{Green-book} or \cite{Tao-book}.

Let $N$ be an integer, $f \colon [N] \to \CC$ and $s \in \NN$. We denote by $\Pi(N)$ the set of $s$-dimensional parallelepipeds contained in $[N]$:
$$
\Pi(N) := \set{ \vec n \in \ZZ^{s+1} }{ 1 \omega \cdot \vec n \in [N] \text{ for all } \omega \in \{0,1\}^s}
$$
where we use the shorthand $x \cdot y = \sum_i x_i y_i$ and $1\omega = (1, \omega_1,\omega_2, \dots)$ so that $ 1 \omega \cdot \vec n = n_0 + \omega_1 n_1 + \dots + \omega_s n_s$. The {\em uniformity norm of $f$ of order $s$} is defined by
\begin{equation}
\label{eq:def-U^s}
\norm{f}_{U^s[N]}^{2^s} := \EEE_{\vec n \in \Pi(N)} \prod_{\omega \in \{0,1\}^{s}} \conj^{\abs{\omega}} f(1\omega \cdot \vec n),
\end{equation}
where $\conj$ is the complex conjugation. One can show that $\norm{\cdot}_{U^s[N]}$ is well defined (i.e.\ the right hand side of \eqref{eq:def-U^s} is always real and positive) and is indeed a norm for $s \geq 2$. (For $s = 1$, $\norm{f}_{U^1[N]} = \abs{ \EE_{n < N} f(n) }$ is a seminorm).  Additionally, we have the nesting property 
\begin{equation}
\label{eq:nesting}
	\norm{f}_{U^1[N]} \ll \norm{f}_{U^2[N]} \ll \norm{f}_{U^3[N]} \ll \dots, 
\end{equation}
as well as phase invariance
\begin{equation}
	\norm{\varphi f}_{U^s[N]} = \norm{f}_{U^s[N]},\ \text{where } \varphi(n) = e(p(n)),\ p \in \RR[x],\ \deg p < s.
\end{equation}

As a direct consequence of the above properties we have
$$
	\sup_{\substack{ p \in \RR[x] \\ \deg p < s}} \abs{ \EEE_{n < N} f(n) e(p(n))} \leq
		\sup_{\substack{ p \in \RR[x] \\ \deg p < s}} \norm{ f \cdot e \circ p }_{U^s[N]} =
	\norm{ f }_{U^s[N]}
$$
In particular, Theorem \ref{thm:oscillating=>fully-oscillating} follows immediately from Theorem \ref{thm:oscillating=>Gowers-uniform}. 

It is often useful to bound the uniformity norms in terms of other norms which are easier to study. The case $s=2$ is particularly simple, since $\norm{f}_{U^2[N]}$ is essentially equivalent to the supremum norm of the Fourier transform (see e.g.\ \cite[Prop.\ 2.2]{GreenTao-2008}):
\begin{equation}
\label{eq:U2-vs-Fourier}
\sup_{\alpha \in \RR} \abs{ \EEE_{n<N} f(n) e(-n\alpha) } \ll \norm{f}_{U^2[N]} \ll
\sup_{\alpha \in \RR} \abs{ \EEE_{n<N} f(n) e(-n\alpha) }^{1/2} \norm{f}_{\infty}^{1/2}.
\end{equation}
 For $p \geq 1$, we denote by $\norm{\cdot}_{L^p[N]}$ the $L^p$-norm corresponding to the normalised counting measure on $[N]$. It is shown in \cite{EisnerTao-2012} that for any $f \colon [N] \to \CC$ we have 
\begin{equation}
\label{eq:962}
\norm{f}_{U^s[N]} \ll \norm{f}_{L^p[N]},
\end{equation}
where $p = p(s) = 2^s/(s+1)$.

Because the functions we work with behave well with respect to the expansion in base $q$, it is more convenient to work with the norms $\norm{\cdot}_{U^s[N]}$ when $N$ is a power of $q$. The following lemma is standard.
\begin{lemma}\label{lem:wlog-N=q^L-A}
	For any $s \geq 2$ there exists a constant $\lambda > 0$ such that for any $\eta > 0$ the following holds. Let $f \colon \NN_0 \to \UU$ be any sequence and $N,M \geq 0$ be integers with $\eta M < N \leq M$. 
	Then we have $\norm{f}_{U^s[N]} \ll_{\eta} \norm{f}_{U^s[M]}^\lambda$.
\end{lemma}
\begin{proof}
	Let $p = 2^s/(s+1)$, $\delta = \norm{f}_{U^s[M]}$. By the nesting property  \eqref{eq:nesting} and \eqref{eq:U2-vs-Fourier} we have
	$$\delta = \norm{f}_{U^s[M]} \gg \norm{f}_{U^2[M]} \gg \sup_{\alpha \in \RR} \abs{ \EEE_{n<M} f(n) e(-n\alpha) } \gg 1/\sqrt{M},$$
	where the last inequality follows from Parseval identity. Put $H = \floor{\delta M}$; we have $\delta M \leq H \ll \delta M$. We will identify $[M]$ with $\ZZ/M\ZZ$ and using Fourier analysis on the latter group. Put $h_1 = 1_{[N]} \ast 1_{[H]}$ and $h_2 = 1_{[N]} - h_1$. Then by a standard application of the Cauchy--Schwarz inequality and the Parseval identity we have
	$$
		h_1(n) = \sum_{k<M} \hat h_1(k) e(kn/M), \text{ with } \sum_{k < M} \abs{ \hat h_1(k) } \ll 1/\delta^{1/2}.
	$$ 
It follows from the construction that $\abs{\operatorname{supp} h_2} \ll \delta M$ and $\norm{h_2}_\infty \ll 1$, whence $\norm{h_2}_{L^p[M]} \ll \delta^{1/p}$.
Using the phase-invariance and the triangle inequality we conclude that 
	$$
	\norm{f}_{U^s[N]} \ll \norm{1_{[N]} f}_{U^s[M]} \leq \sum_{k<M} \abs{ \hat h_1(k)} \norm{f}_{U^s[M]} + \norm{h_2}_{L^p[M]} \ll \delta^{1/2} + \delta^{1/p}.
	$$
	This gives the claim with the (non-optimal) constant $\lambda = \min(1/p,1/2)$.
\end{proof}
For $q$-multiplicative sequences we have the following stronger version of the above lemma.
\begin{lemma}\label{lem:wlog-N=q^L}
	For any $s \geq 2$ there exists a constant $\lambda > 0$ such that for any $\eta > 0$ the following holds. Let $f \colon \NN_0 \to \UU$ be a $q$-multiplicative sequence and $N,M \geq 0$ be integers with $\eta M < N \leq M/\eta$. 
	Then we have $\norm{f}_{U^s[N]} \ll_{\eta} \norm{f}_{U^s[M]}^\lambda$.
\end{lemma}
\begin{proof}
Note that for any integer $L \geq 0$ we have $\norm{f}_{U^s[q^{L+1}]} \leq q \norm{f}_{U^s[q^{L}]}$. Indeed, letting $P_a = a q^{L} + [q^L]$ ($0 \leq a < q$), which form a partition of $[q^{L+1}]$, and using the triangle inequality and $q$-multiplicativity we find that
$$
	\norm{f}_{U^s[q^{L+1}]}
	\leq \sum_{a =0}^{q-1} \norm{ 1_{P_a} f}_{U^s[q^{L+1}]} \leq q \norm{f}_{U^s[q^L]}. 
$$
	Pick $L \geq 0$ and $C = O_\eta(1)$ such that $N \leq q^{L+C}$ and $M \geq q^L$. Then
$$
	\norm{f}_{U^s[N]} \ll_\eta \norm{f}_{U^s[q^{L+C}]}^{\lambda} \ll_{\eta} \norm{f}_{U^s[q^{L}]}^\lambda \ll_{\eta} \norm{f}_{U^s[M]}^{\lambda^2},
$$	
where $\lambda$ is the constant from Lemma \ref{lem:wlog-N=q^L-A}.
\end{proof}

The criterion for $U^2$-uniformity of a $q$-multiplicative sequence can be somewhat simplified, as shown by the following lemma.

\begin{lemma}\label{lem:U2-uniform}
Let $f \colon \NN_0 \to \UU$ be a $q$-multiplicative sequence and suppose that
\begin{equation}\label{eq:39:00}
	\EEE_{n<N} f(n) e(n\a) \to 0 \text{ as } N \to \infty \text{ for any } \a \in \RR.
\end{equation}
Then the convergence in \eqref{eq:39:00} is uniform with respect to $\alpha$, that is
\begin{equation}\label{eq:39:01}
	\sup_{\alpha \in \RR} \abs{ \EEE_{n<N} f(n) e(n\a)} \to 0 \text{ as } N \to \infty.
\end{equation}
\end{lemma}
\begin{proof}
Suppose that the claim was false and pick $\e \in (0,1/2)$ and sequences $N_i$, $\a_i$ ($i \geq 0$) such that $N_i \to \infty$ as $i \to \infty$ and 
\begin{equation}\label{eq:39:02}
	\abs{ \EEE_{n<N_i} f(n) e(n\a_i)} \geq \e
\end{equation}
for each $i \geq 0$. Passing to a subsequence we may additionally assume that $\a_i \to \a$ as $i \to \infty$ for some $\a \in \RR$. For any $L \geq 0$ and any $i$ large enough that $N_i > 10 q^{L}/\e$,  partitioning the average in \eqref{eq:39:02} into intervals of length $q^L$ and a leftover part of shorter length, and using the triangle inequality we conclude that
\begin{equation}\label{eq:39:03}
	\abs{ \EEE_{n< q^L} f(n) e(n\a_i)} \geq \e/2.
\end{equation}
Letting $i \to \infty$, we conclude that 
\begin{equation}\label{eq:39:04}
	\liminf_{L \to \infty} \abs{ \EEE_{n< q^L} f(n) e(n\a)} \geq \e/2,
\end{equation}
contradicting \eqref{eq:39:00}.
\end{proof}

We record a variant of the generalised von Neumann theorem (see \cite{Gowers-2001} for a special yet representative case and \cite{GreenTao-2010} for a considerably more general statement). A system of affine forms $\Psi = (\psi_1,\dots,\psi_t) \colon \ZZ^d \to \ZZ^t$ is said to have {\em finite complexity} if there is no non-trivial linear combination $\alpha \psi_i + \beta \psi_j$ ($\alpha,\beta \in \RR$ with $(\alpha,\beta) \neq (0,0)$ and $1 \leq i < j \leq t$) which is a constant function.
\begin{theorem}\label{thm:von-Neumann}
	Suppose that $\Psi = (\psi_1,\dots,\psi_t) \colon \ZZ^d \to \ZZ^t$ is a system of affine forms with finite complexity, and let $f_i \colon [N] \to \CC$ be functions bounded in absolute value by $1$ for $1 \leq i \leq t$. Then there exists $s \geq 1$, dependent only on $\Psi$, such that 
	$$
		\EEE_{\vec n \in [N]^d} \prod_{i = 1}^t (1_{[N]}f_i)( \psi_i( \vec n)) \ll \min_{1 \leq i \leq t} \norm{f_i}_{U^s[N]}.
	$$ 
\end{theorem}
(Above, we denote by $1_{[N]}f_i$ the function which coincides with $f$ on $[N]$ and is identically $0$ elsewhere.) The optimal value of $s$  has been the subject of detailed investigation, see \cite{GowersWolf-2010} as well as \cite{Gowers-2001} and \cite{GreenTao-2010}, but its value will be largely irrelevant to us. 

It follows from the above discussion that polynomial phases $f(n) = e(p(n))$ with $p \in \RR[x]$ and $\deg p < s$ have maximal possible Gowers norms among $1$-bounded sequences: $\norm{f}_{U^s[N]} = 1$. Roughly speaking, an inverse theorem for Gowers norms asserts that, conversely, if $f \colon [N] \to \CC$ is bounded by $1$ and the norm $\norm{f}_{U^s[N]}$ is large then $f$ resembles a polynomial phase. Depending on how one interprets ``large'' in the previous statement, there are several regimes which may be considered.

In the ``100\% variant'' (following the terminology of Tao \cite{Tao-book}), we assume that $\norm{f}_{U^s[N]} = 1$. It is then a simple exercise to show that $f(n) = e(p(n))$ for all $n$, where $p \in \RR[x]$ and $\deg p < s$ as above. In principle, this (together with some continuity arguments) would suffice to establish our main result, but the reasoning is more elegant and intermediate steps yield stronger conclusions when we instead use the ``99\% variant''. The following result can be obtained by the methods developed in \cite{AlonKaufmanKrivelevichLitsynRon-2003}, for similar result see also \cite[Lem.\ 4.5]{TaoZiegler-2010}; see \cite[p.\ 1234]{GreenTaoZiegler-2012} or \cite[Prop.\ 1.5.1, Ex.\ 1.6.21]{Tao-book} for reference to (essentially) this statement.
\begin{theorem}\label{thm:inverse-99}
	There exists a function $h \colon \RR_{>0} \to \RR_{>0}$ such that $h(\e) \to 0$ as $\e \to 0$ and the following holds.
	
	Suppose that $\e>0$ and $f \colon [N] \to \CC$ is bounded by $1$ in absolute value and $\norm{f}_{U^s[N]} \geq 1 - \e$. Then there exists a polynomial phase $g(n) = e(p(n))$ where $p \in \RR[x]$ with $\deg p < s$ such that $\norm{f - g}_{L^1[N]} \leq h(\e)$.
\end{theorem}

Lastly, we mention the ``1\% variant'' of the inverse theorem, due to Green, Tao and Ziegler \cite{GreenTaoZiegler-2012} (see also \cite{BergelsonTaoZiegler-2010}, \cite{TaoZiegler-2010}, \cite{TaoZiegler-2012}). It asserts, in rough terms, that if $\norm{f}_{U^s[N]} \geq \e$ for some $\e > 0$ then $f$ correlates with a $(s-1)$-step nilsequence of bounded complexity. Because this result is beside the scope of the present paper, we do not go into more detail, and refrain even from precisely defining the terms appearing in the previous statement. 

\makeatletter{}\section{Recursive relations}

\begin{convention*} Throughout this and subsequent sections, we treat the basis $q \geq 2$ and the order of the Gowers norm $s \geq 2$ as fixed, and we allow all constants --- implicit and explicit alike --- to depend on both $q$ and $s$ without further mention. In particular, wherever we write $O(\cdot)$, we really mean $O_{q,s}(\cdot)$.
\end{convention*}

 For a function $f \colon \NN_0 \to \CC$, $\bb r \in \NN_0^{2^s}$ and $L \in \NN_0$ we define
$$A(f,\bb r,L) := \EEE_{\vec n \in \Pi(q^L)} \prod_{\omega \in \{0,1\}^{s}} \conj^{\abs{\omega}} f(1\omega \cdot \vec n + r_\omega),$$
where $\conj$ is the complex conjugation. Hence, $A(f,\bb 0,L) = \norm{f}_{U^s[q^L]}^{2^s}$. Our goal is, roughly speaking, to show that $A(f,\bb r,L)$ are small if $f$ is $q$-multiplicative and oscillating. Recall that by Lemma \ref{lem:wlog-N=q^L}, in order to prove Theorem \ref{thm:oscillating=>Gowers-uniform} it is enough to deal with the case $N = q^L$.

For any function $f \colon \NN_0 \to \CC$ and $l \geq 0$, denote by $S^l f$ the $q$-multiplicative function given by $S^lf(n) = f(q^ln)$. The following approximate recurrence expresses $A(f,\bb r, L)$ in terms of shorter averages $A(S^lf, \bb r', L-l)$. To make notation slightly less burdensome, we shall use Iverson bracket notation, where for a statement $\phi$, $\ifbra{\phi} = 1$ if $\phi$ is true and $\ifbra{\phi} = 0$ if $\phi$ is false. 

\begin{lemma}\label{lem:recurrence}
	Let $f \colon \NN_0 \to \UU$ be $q$-multiplicative, $\bb r \in \NN_0^{2^s}$ and $L,l \in \NN$ with $l \leq L$. Then
	\begin{equation}\label{eq:50:00}
	A(f,\bb r, L) = \sum_{\bb r' \in \NN_0^{2^s}} A(S^l f, \bb r', L-l) W^{(l)}(f,\bb r, \bb r') + O(q^{-(L-l)}),
	\end{equation}
	where only finitely many terms in the sum are non-zero and
	\begin{equation}\label{eq:50:02}
	 W^{(l)}(f,\bb r, \bb r') 
	 = \EEE_{\vec e \in [q^l]^{s+1}}
	\prod_{\omega \in \{0,1\}^s} \conj^{\abs{\omega}}f(1\omega \cdot \vec e \bmod{q^l}) \ifbra{ \delta^{(l)}(\bb r, \vec e) = \bb r'},
	\end{equation}
	with $\delta^{(l)}(\bb r, \vec e) \in \NN_0^{2^s}$ given by
	\begin{equation}\label{eq:50:01}
	\delta^{(l)}(\bb r, \vec e)_{\omega} = \floor{\frac{1\omega \cdot \vec e + r_\omega}{q^l}}.
	\end{equation}
\end{lemma}
\begin{proof}
	This follows by essentially the same argument as \cite[Lemma 2.1]{Konieczny-2017+} (see also  \cite[Lemma 3.1]{Konieczny-2017+}).
\end{proof}

\begin{remark*}
	We will be interested only in $\bb r$ such that $0 \leq r_\omega \leq s$ for all $\omega \in \{0,1\}^s$. Note that for such $\bb r$, if $\delta^{(l)}(\bb r, \vec e) = \bb r'$ for some  $l \geq 0$ and $\vec e \in [q^l]^{s+1}$ then also 
	$$r_\omega' \leq \floor{ \frac{(s+1)(q^l-1) + s}{q^l} } = s \text{ for all $\omega \in \{0,1\}^s$. }$$
\end{remark*}

The coefficients $W^{(l)}(f,\bb r, \bb r')$ given by \eqref{eq:50:02} are reminiscent of Gowers uniformity norms. We will be especially interested in $W^{(l)}(f,\bb 0, \bb 0')$, which can more simply be written as
	\begin{equation}\label{eq:50:90}
\PPP_{\vec e \in [q^l]^{s+1}} \bra{ \sum_{i=0}^s e_i < q^l } \EEE_{\vec e \in [q^l]^{s+1}}
	\bra{\prod_{\omega \in \{0,1\}^s} \conj^{\abs{\omega}}f(1\omega \cdot \vec e ) \middle| \sum_{i=0}^s e_i < q^l }.
	\end{equation}
Note that the average in \eqref{eq:50:90} is not quite the same as the average defining $\norm{f}_{U^s[q^l]}^{2^s}$ in \eqref{eq:def-U^s}.
Indeed, while the two averages contain the same collection of terms, there is a slight difference caused by the existence of degenerate cubes. For instance, if $s = 2$ then the terms $f(0)\bar f(1)^2 f(2)$ and $f(0)^2 \bar f(0)^2$ appear with the same multiplicity in the average in \eqref{eq:50:90}, but with different multiplicities in the average in \eqref{eq:def-U^s}. In the following lemma we circumvent this technical issue.

\begin{lemma}\label{lem:E-to-Gowers} 
	Let $N \geq 0$, $\e > 0$ and $f \colon \NN_0 \to \UU$. Then the following conditions are equivalent:
\begin{enumerate}
\item\label{cond:00:A} $\norm{f}_{U^s[N]} \geq 1 - O(\e)$,
\item\label{cond:00:B} $ \Re {\EEE_{\vec e \in [N]^{s+1}}
	\bra{\prod_{\omega \in \{0,1\}^s} \conj^{\abs{\omega}} f(1\omega \cdot \vec e) \middle| \sum_{i=0}^{s} e_i < N}}
	 \geq 1 - O(\e)$.
\end{enumerate}
\end{lemma}
\begin{remark*}
The equivalence in the above lemma should be understood in the following sense: If  \eqref{cond:00:A} holds with a certain constant $C_1$ implicit in the $O$-notation, then \eqref{cond:00:B} holds with some other implicit constant $C_2 = C_2(C_1)$, and conversely if \eqref{cond:00:B} holds with implicit constant $C_2'$ then \eqref{cond:00:A} holds with some other implicit constant $C_1' = C_1'(C_2')$. We do not claim (nor expect) that there exists a choice of the implicit constants $C_1^*$, $C_2^*$ such that the two resulting statements are precisely equivalent.
\end{remark*}
\begin{proof}
For an $(s+1)$-tuple of integers $\vec n = (n_i)_{i=0}^s$, denote by $F(\vec n) = \bra{ 1\omega \cdot \vec n}_{\omega \in \{0,1\}^s}$ the corresponding parallelepiped. For any $j$ with $1 \leq j\leq s$, we may consider the ``reflection'' along the $j$-th principal axis which maps $F(\vec n)$ to $F(\vec n')$ where $n_0' = n_0 + n_j,\ n_j' = -n_j$ and $n_i' = n_i$ for $i \neq 0,j$. Note that such reflection does not alter the set of vertices of the parallelepiped, merely their ordering.

Consider an equivalence relation $\sim$ on the set of all parallelepipeds (with integer coordinates) where two parallelepipeds $F(\vec n)$ and $F(\vec n')$ are equivalent if and only if one can be obtained from the other by an even number of reflections. Denote by $\cP$ the family of equivalence classes of parallelepipeds with vertices in $[N]$, and for an equivalence class $P := [F(\vec n)]_{\sim} \in \cP$ let $\bar{P}$ denote the equivalence class of the reflection of $F(\vec n)$ along one of the principal axes, and let $f(P) := \prod_{\omega} \conj^{\abs{\omega}} f(1\omega \cdot n_\omega)$. It is standard to check that these definitions are well posed and $f([\bar P]) = \overline{f([P])}$. 

Denote by $\mu$ (resp.\ $\nu$) the probability measure on $\cP$ such that $\mu(P)$ is the probability that $[F(\vec n)]_{\sim} = P$ if the coordinates $n_i$ are chosen independently uniformly at random from $(-N,N) \cap \ZZ$ (resp. from $[0,N) \cap \ZZ$) subject to the condition that all vertices of $F(\vec n)$ lie in $[N]$. This is set up so that 
\begin{equation}
\label{eq:150:01}
\norm{f}_{U^s[N]}^{2^s} = 
{\EEE_{\vec n \in \Pi(N) }
	\bra{\prod_{\omega \in \{0,1\}^s} \conj^{\abs{\omega}} f(1\omega \cdot \vec n) }} 
	= \sum_{P \in \cP } \mu(P) f(P) 
\end{equation}
	and similarly
\begin{equation}
\label{eq:150:02}
{\EEE_{\vec n \in [N]^{s+1} }
	\bra{\prod_{\omega \in \{0,1\}^s} \conj^{\abs{\omega}} f(1\omega \cdot \vec n) \middle| \sum_{i=0}^{s} n_i < N}} 
	= \sum_{P \in \cP} \nu(P) f(P).
\end{equation}

Additionally, it follows from the construction that $\mu(P) = \mu(\bar P)$. Unfortunately, $\nu$ does not enjoy this symmetry, so we introduce $\tilde \nu(P) := (\nu(P) + \nu(\bar P))/2$.

For any parallelepiped $F(\vec n)$ there is precisely one parallelepiped $F(\vec n')$ which can be obtained from $F(\vec n)$ by reflections such that $n_i' \geq 0$ for all $1 \leq i \leq s$. Conversely, since the reflections commute, every $\vec n'$ corresponds to $2^r$ choices of $\vec n$ where $0 \leq r \leq s$ depends on the number of $i$ such that $n_i' = 0$. It follows that there exists a constant $C$ (actually, we can take $C = 2^s$) such that for all $P \in \cP$ we have
\begin{equation}
\label{eq:150:05}
 1/C \leq { \mu(P)}/{ \tilde\nu(P) } < C.
\end{equation}
Having set up the notation, we can equivalently write \eqref{cond:00:A} as
\begin{equation}
\label{eq:150:06}
	\sum_{P \in \cP} \mu(P) \bra{ 1-\Re f(P)} = O(\e).
\end{equation}
Note that we can freely replace $f(P)$ with $\Re f(P)$ because of symmetry of $\mu$. Similarly, \eqref{cond:00:B} is equivalent to
\begin{equation}
\label{eq:150:07}
	\sum_{P \in \cP} \tilde \nu(P) \bra{ 1-\Re f(P)} = O(\e),
\end{equation}
where we similarly note that we could replace $\nu$ with $\tilde\nu$ because $\Re f(P) = \Re f(\bar P)$. It remains to note that conditions \eqref{eq:150:06} and  \eqref{eq:150:07} are equivalent (up to the a factor of $C$ in the implicit constant) because of \eqref{eq:150:05}. 
\end{proof}

The absolute values of the weights $W^{(l)}(f,\bb r, \bb r')$ appearing in \eqref{eq:50:00} for any choice of $f$ and $\bb r$ sum to at most $1$; indeed, by the triangle inequality we have
\begin{equation}
\label{eq:150:57}
	\sum_{\bb r' \in \NN_0^{2^s}} \abs{ W^{(l)}(f,\bb r, \bb r')} 
	\leq \sum_{\bb r' \in \NN_0^{2^s}} \PP_{\vec e \in [q^l]^{s+1}} \bra{\delta^{(l)}(\bb r, \vec e) = \bb r'} = 1.
\end{equation}
Recall that by Lemma \ref{lem:E-to-Gowers}, an improvement over the trivial bound $ \abs{ W^{(l)}(f,\bb r, \bb r')} \leq \PP\bra{\delta^{(l)}(\bb 0, \vec e) = \bb 0}$ is equivalent to a comparable improvement over the trivial bound $\norm{f}_{U^s[q^l]} \leq 1$.
The following proposition shows how such an improvement translates into a bound on $A(f,\bb r, L)$. 

\begin{proposition}\label{prop:E-bound=>A-bound}
	There exist constants $l_0 \in \NN$ and $c > 0$ such that the following is true. Suppose that $f \colon \NN_0 \to \UU$ is $q$-multiplicative and that $K_i \geq 0$ and $\e_i > 0$ ($i \geq 0$) are sequences with $K_0 = 0$, $L_i := K_{i+1}-K_i \geq l_0$ and
	\begin{equation}\label{eq:51:50}
	\norm{S^{K_i} f}_{U^s[q^{L_i}]} \leq  1 - \e_i \text{ for all } i \geq 0.
	\end{equation}	 
Then for any $M \geq 0$, any $K \geq K_M$ and any $\bb r \in \NN_0^{2^s}$ with $0 \leq r_\omega \leq s$ for all $\omega \in \{0,1\}^s$ it holds that
\begin{equation}\label{eq:51:01}
	 A(f,\bb r, K) \ll \exp\bra{-c \sum_{i = 0}^{M-1} \e_i}.	
 	\end{equation}	 
\end{proposition}
\begin{proof}
	By Lemma \ref{lem:E-to-Gowers}, rescaling $\e_i$ by a constant factor we may assume that for all $i \geq 0$ we have
	\begin{equation}\label{eq:51:00}
	\Re{\EEE_{\vec e \in [q^{L_i}]^{s+1}}
	\bra{\prod_{\omega \in \{0,1\}^s} \conj^{\abs{\omega}} S^{K_i} f(1\omega \cdot \vec e \bmod{q^{L_i}}) \middle| \delta^{(L_i)}(\bb 0, \vec e) = \bb 0}}
	 \leq 1 - \e_i
	\end{equation}	
	Applying the recursive formula \eqref{eq:50:00} from Lemma \ref{lem:recurrence} with $L = K$ and $l = K-K_M$ we see that it will suffice to show the estimate \eqref{eq:51:01} with $K=K_M$. Fix the value of $M$ and put	
	\begin{align}
	B_1(i) &= \max_{\bb r} \abs{ A(S^{K_i}f,\bb r,K_M-K_i) \label{eq:51:02}}\\
	B_0(i) &= {A(S^{K_i}f,\bb 0,K_M-K_i)} = \norm{S^{K_i}f}_{U^s[q^{K_M-K_i}]}^{2^s}, \label{eq:51:03}
	\end{align}
	where the maximum runs over $\bb r \in \NN_0^{2^s}$ with $0 \leq r_\omega \leq s$. Note that $0 \leq B_0(i) \leq B_1(i) \leq 1$ for all $i \geq 0$, and that crucially in \eqref{eq:51:03} there is no absolute value. Our aim is to show that $B_1(0) \ll \exp\bra{-c \sum_{i = 0}^{M-1} \e_i}$.
	
	Provided that $l_0$ is large enough, it follows directly from \eqref{eq:50:01} there exists a constant $\rho > 0$ such that 
	\begin{equation}\label{eq:51:04}
	\PPP_{\vec e \in [q^l]^{s+1}}\bra{\delta^{(l)}(\bb r, \vec e) = \bb 0} \geq \rho
	\end{equation}
	for all $l \geq l_0$ and all $\bb r$ as above. Hence, applying \eqref{eq:50:00} with $L = K_M - K_{i}$ and $l = L_i$ for any $i \geq 0$ we obtain 
\begin{align}
	B_1(i) &\leq (1-\rho) B_1(i+1) + \rho B_0(i+1) + O(q^{-(K_M - K_{i+1})}) \label{eq:51:05}\\
	B_0(i) &\leq (1-\rho\e_{i}) B_1(i+1) + O(q^{-(K_M - K_{i+1})}). \label{eq:51:06}
\end{align} 
Combining the two estimates \eqref{eq:51:05} and \eqref{eq:51:06}, we conclude that 
\begin{align}
	B_1(i) &\leq (1-\rho^2\e_{i+1}) B_1(i+2) + O(q^{-(K_M - K_{i+2})}). \label{eq:51:07}
\end{align} 

Let $k \leq M$ be an index to be optimised in the course of the argument.
Applying \eqref{eq:51:07} to even $i$ with $0 \leq i < k-2$ and using the bound $1-x \leq e^{-x}$, we conclude that
\begin{align}
	B_1(0) & \ll \exp\bra{ - \rho^2 \sum_{j < k/2 } \e_{2j} } +  O(q^{-(K_M-K_k)}). \label{eq:51:08}
\end{align} 
Similarly, taking odd $i$ in the same range $0 \leq i < k-2$  we obtain
\begin{align}
	B_1(0) \ll B_1(1) \ll \exp\bra{ -\rho^2 \sum_{j < k/2 } \e_{2j+1} } + O(q^{-(K_M-K_k)}). \label{eq:51:09}
\end{align}

Combining \eqref{eq:51:08} and \eqref{eq:51:09} and using the estimate $K_M - K_k \geq (M-k)l_0$ we conclude that
\begin{align}
	B_1(0) \ll \exp\bra{ - \frac{\rho^2}{2} \sum_{i < k } \e_{i}} + \exp\bra{- (M-k)l_0 \log q }.  \label{eq:51:10}
\end{align}

Pick $k$ such that $\sum_{i < k } \e_{i} = \frac{1}{2} \sum_{i<M} \e_i + O(1)$ (which is always possible since $\e_i < 1$). Then $M-k \geq \frac{1}{2} \sum_{i<M} \e_i - O(1)$. Hence, \eqref{eq:51:10} yields
\begin{align}
	B_1(0) \ll \exp\bra{ - c \sum_{i < M } \e_{i}},  \label{eq:51:11}
\end{align}
where we may further specify $c = \rho^2/4$.
\end{proof}

It follows from the above result that, roughly speaking, for a $q$-multiplicative sequence $f \colon \NN_0 \to \UU$ the uniformity norms $\norm{f}_{U^s[N]}$ are small unless the norms $\norm{S^K f}_{U^s[q^L]}$ are almost always close to $1$ (for, say, bounded $L$ and arbitrary $K$). In the latter case, inspired by the terminology introduced by A.\ Granville \cite{Granville-2009}, we will informally say that $f$ is ``pretentious'' (since it ``pretends'' to be a polynomial phase). In the subsequent section, we shall study such ``pretentious'' functions in considerable detail.

\makeatletter{}\section{Pretentious $q$-multiplicative functions}

In this section, we obtain a description of ``pretentious'' $q$-multiplicative sequences, which are introduced towards the end of the previous section. As discussed in Section \ref{sec:Background}, $1$-bounded sequences with near-maximal Gowers norms behave approximately like polynomial phases. In the following proposition, we show that the only $q$-multiplicative sequences which behave approximately like polynomial phases are essentially equal to linear phases.

\begin{proposition}\label{prop:rigidity}
	There exist constants $\e_0 > 0$, $\delta_0 > 0$ and $L_0 > 0$ such that the following holds. 

	Suppose that $f \colon \NN_0 \to \UU$ is $q$-multiplicative, $p \in \RR[x]$ with $\deg p < s$ and
\begin{equation}\label{eq:144:00}
	f(n) = e(p(n) + \OO(\e)) \quad \text{for $\delta$-almost all } n < q^L,
\end{equation}
	for some $\e < \e_0$, $\delta < \delta_0$ and $L \geq L_0$. Then there exist $\alpha \in \RR$ and $b \in \QQ[x]$ such that $b(n) \bmod{1}$ is periodic with a period $q^{O(1)}$ and
\begin{equation}\label{eq:144:09}
	f(n) = e(\alpha n + \beta + b(n) + O(\e)) \quad \text{for $\delta$-almost all } n < q^L.
\end{equation}
	
\end{proposition}
\begin{proof}
	Write $p(x) = \sum_{j = 0}^{s-1} \alpha_j x^j$. We first show by a downwards induction on $d \in \{2,3,\dots,s-1\}$ that $\alpha_d$ takes the form $\alpha_d = \eta_d + a_d/Q_d$ where $\eta_d = O(\e)/q^{dL}$, $a_d \in \ZZ$ and $Q_d$ is recursively given by $Q_s = 1$ and $Q_{d} = d! q^{(d-1)d/2}Q_{d+1}$. Letting $\alpha_s = 0$, the statement is vacuously true for $d = s$.

	For a set $I \subset \{0,1,\dots,L-1\}$, denote by $\cF_I$ the set of integers $n < q^L$ with expansions $n = \sum_{l<L} n^{(l)}q^l$ such that $n^{(l)} = 0$ for all $l \in I$, and let $m_I = \sum_{l \in I} q^l$. Let $J \subset \{0,\dots,L-1\}$ be any fixed set with $\abs{J} = d$. Then for $2^d q^{d}\delta$-almost all $n \in \cF_J$ it holds that
	\begin{align}\label{eq:144:21}
		1 = \prod_{I \subset J} \conj^{\abs{I}} \bra{f(n) \prod_{l \in I} f(q^l) }
		 &= \prod_{I \subset J} \conj^{\abs{I}} f(n + m_I) \\&=
		e\bra{ \sum_{I \subset J} (-1)^{\abs{I}} p(n+m_I) + O(\e) }. 
	\end{align}	 
	(Note that the first equality holds by inclusion-exclusion principle and is only true when $d \geq 2$. The last equality is the only one which does not holds for all $n$, and it fails for at most $2^d \delta q^L$ choices of $n$ by the union bound.)
	Letting $$q(n) = \sum_{I \subset J} (-1)^{\abs{I}} p(n+m_I),$$ we conclude that $\fpa{q(n)} = O(\e)$ for $2^d q^{d}\delta$-almost all $n \in \cF_J$. Note also $\deg q < s$ so 
	\begin{equation}\label{eq:144:45}	
	\sum_{j = 0}^{s} (-1)^j \binom{s}{j} q(n+jh) = 0 \text{ for any $h \in \ZZ$}.
	\end{equation}
	It is elementary to show that given $n \in \cF_I$ there exist $\gg q^L$ choices of $h \in \ZZ$ such that $n + jh \in \cF_I$ for $1 \leq j \leq s$. 	
	 (One way to construct such $h$ is as follows. Set $k = \floor{\log_q s +10}$. For each $l \in J \cup \{L-1\}$, if $J \cap [l-k,l) \neq \emptyset$ let $l' = \max J \cap [l-k,l)$ the element of $J$ preceding $l$; else let $l' = \max(l-k,0)$. If $n^{(l-1)} \leq q/2$ then set $h^{(l)} = h^{(l-1)} = \dots h^{(l'+1)} = 0$. Otherwise, if $n^{(l-1)} > q/2$ then set $h^{(l)} = h^{(l-1)} = \dots h^{(l'+1)} = q-1$. The remaining digits $h^{(l)}$ for $0 \leq l < L$ are arbitrary. If $h \geq q^L/2$, replace it with $h - q^L$. We have specified fewer than $ d (\log_q s + 10)$ digits of $h$, so the number of $h$ thus constructed is at least $ q^{L - 10 d}/s^d \gg q^{L}$.)
	 
	 On the other hand, for $O(\delta)$-almost all $h \in [-q^L/2, q^L/2]$ it holds for each $1 \leq j \leq s$ that if $n+jh \in \cF_I$ then $\fpa{q(n+jh)} = O(\e)$. Hence, if $\delta_0$ is small enough in terms of implicit constants above, there is at least one $h$ such for each $1 \leq j \leq s$ it holds that $n+jh \in \cF_I$ and $\fpa{q(n+jh)} = O(\e)$. It follows from \eqref{eq:144:45} that $\fpa{q(n)} = O(\e)$ for all $n \in \cF_J$, provided that $\delta_0$ is small enough (in terms of $s$). In particular,  	\begin{equation}\label{eq:144:21b}
	\fpa{q(0)} = \fpa{ \sum_{I \subset J} (-1)^{\abs{I}} p(m_I) } = O(\e).
	\end{equation}	
	
	Using the inductive assumption we see that
	\begin{equation}\label{eq:144:22}
		 \sum_{I \subset J} (-1)^{\abs{I}} p(m_I) 
		 = d! q^{\sum_{l \in J} l} \alpha_d 
		 + \sum_{j > d} \frac{a_j}{Q_j} \sum_{I \subset J} (-1)^{\abs{I}} m_I^j + \sum_{j > d} \eta_j \sum_{I \subset J} (-1)^{\abs{I}} m_I^j.
	\end{equation}
	Note that the terms of degree $< d$ cancel out by elementary algebra and the inclusion-exclusion principle.
	
	In \eqref{eq:144:22}, the first term is the main contribution, the second one can be ignored after multiplying by $Q_{d+1}$ and the third one is a negligible error term and can be bounded by $O(\e)$. Combining \eqref{eq:144:22} with \eqref{eq:144:21b} we thus conclude that
	\begin{equation}\label{eq:144:23}
		q^{k(J)} Q_d\alpha_d = O(\e) \bmod{1},
	\end{equation}
	where $k(J) = \sum_{l \in J} l - d(d-1)/2$. For any choice of $J$ we have $0 \leq k(J) \leq d L - d^2$, and conversely for any $k$ in this range there exists $J$ such that $k(J) = k$. Hence, 
	\begin{equation}\label{eq:144:24}
		q^{k} Q_d\alpha_d = O(\e) \bmod{1} \text{ for any $k$ with $0 \leq k \leq d L - d^2$}.
	\end{equation}
	
	Using \eqref{eq:144:24} inductively for all $k$ in the prescribed range and assuming that $\e_0$ is small enough that the error term in \eqref{eq:144:24} is strictly less than $1/q^2$ we conclude that $Q_d\alpha_d = O(\e/q^{dL})$, whence we have the decomposition $\alpha_d = a_d/Q_d + \eta_d$ with $\eta_d = O(\e/q^{dL})$.
	
	Combining the coefficients $a_d/Q_d$ and $\eta_d$ to form polynomials $b'(n)$ and $\eta'(n)$ respectively, we obtain the decomposition
	\begin{equation}\label{eq:144:60}
		p(n) = \alpha_0 + \alpha_1 n + b'(n) + \eta'(n).
	\end{equation}
	The bound $\fpa{\eta'(n)} = O(\e)$ for $n < q^L$ follows directly from the bounds on the coefficients $\eta_d$. Moreover, $b'(n)$ is a rational polynomial with denominator which divides $Q_2$ and no constant and linear term; in particular $b'(n) \bmod{1}$ is $Q_2$-periodic and $b'(0) = 0$. 
	
	Our next step is to modify $b'(n)$ by adding a linear term in such a way that the resulting function has a period which is a power of $q$. To achieve this goal, we first show that $e(b'(n))$ is $q$-multiplicative. Let $n_1, n_2 \geq 0$ be arbitrary integers having non-zero digits at disjoint sets of positions in base $q$. Pick $n_1', n_2' \geq 0$ such that $n_1'+n_2'$ is minimised, subject to the conditions that $n_1' \equiv n_1 \bmod{Q_2}$, $n_2' \equiv n_2 \bmod{Q_2}$ and $n_1',n_2'$ have non-zero digits at disjoint sets of positions. Note that $n_1', n_2' = O(1)$ and denote the set of positions where non-zero digits appear in $n_1' + n_2'$ by $J$. Repeating the reasoning used to establish \eqref{eq:144:21}, we conclude that for $O(\delta)$-almost all $n \in \cF_J$ it holds that
\begin{equation}
\label{eq:144:80}
	b'(n+n_1'+n_2') - b'(n + n_1') - b'(n+n_2') + b'(n) = O(\e) \pmod{1}.
\end{equation}
	In particular, assuming that $L_0$ is large enough and $\delta_0$ is small enough (in absolute terms), we conclude that \eqref{eq:144:80} holds for at least one $n \equiv 0 \bmod{Q_2}$, meaning that
\begin{equation}
\label{eq:144:81}
	b'(n_1'+n_2') - b'(n_1') - b'(n_2')  = O(\e) \pmod{1}.
\end{equation}
	Since the expression on the left hand side of \eqref{eq:144:81} is rational with denominator $\leq Q_2 = O(1)$, so provided that $\e_0$ is chosen small enough \eqref{eq:144:81} implies an exact equality \begin{equation}
\label{eq:144:82}
	b'(n_1+n_2) = b'(n_1'+n_2') = b'(n_1') + b'(n_2') = b'(n_1) + b'(n_2) \pmod{1}.
\end{equation}
Since $n_1,n_2$ were arbitrary, it follows that $e(b'(n))$ is $q$-multiplicative. 
		Let $P$ be the period of $b'(n) \bmod 1$, and write $P$ as $P = P_0 P_1$ where $P_0$ divides a power of $q$ and $P_1$ is coprime to $q$. 	We have the Fourier expansion:
\begin{equation}\label{eq:144:30}
		e(b'(n)) = \sum_{k_0 = 0}^{P_0-1}  \sum_{k_1 = 0}^{P_1-1} c(k_0,k_1) e( n( k_0/P_0 + k_1/P_1)).  
\end{equation}
The coefficients $c(k_0,k_1)$ can be recovered by the usual formula
$$c(k_0,k_1) = \lim_{N \to \infty} \EEE_{n < N} e(b'(n)) e(-n( k_0/P_0 + k_1/P_1)).$$
On the other hand, since $e(b'(n))$ is $q$-multiplicative, for any $l \in \NN_0$ and $\gamma \in \RR$ we have
\begin{equation}\label{eq:44:31}
\EEE_{n < q^l} e(b'(n) - \gamma n) = \prod_{i < l} \EEE_{a < q} e( b'(a q^i) - \gamma a q^i).
\end{equation}
	In particular, if $c(k_0,k_1) \neq 0$ then for any $a < q$ we have 
\begin{equation}\label{eq:144:10}
b'(a q^l) - a q^l k_1/P_1 \to 0 \pmod{1}  \quad \text{ as $l \to \infty$}.
\end{equation}

We claim that there exists at most one choice of $k_1$ for which \eqref{eq:144:10} holds. Otherwise, there would exist distinct $k_1,k_1' < P_1$ such that $\fpa{ q^l (k_1-k_1')/P_1 } \leq 1/q^2$ for $l$ large enough. This is only possible if the denominator of $(k_1-k_1')/P_1$ divides a power of $q$, which is absurd. Fix the value $k_1$ for which \eqref{eq:144:10} holds (at least one such value exists since $e(b(n)) \neq 0$). Letting $b(x) = b'(x) - (k_1/P_1)x$ we conclude from \eqref{eq:144:30} that $b(n) \bmod 1$ is $P_0$-periodic. To obtain \eqref{eq:144:09}, it remains to put $\alpha = \alpha_1 + k_1/P_1$ and $\beta = \alpha_0$.
\end{proof}

\begin{corollary}\label{cor:pretentious=>linear}
There exist a constant $l_0 \geq 0$ and a function $h \colon \RR_{>0} \to \RR_{>0}$ with $h(\e) \to 0$ as $\e \to 0$ such that the following holds. 

	Suppose that $f \colon \NN_0 \to \UU$ is $q$-multiplicative and $\norm{f}_{U^s[q^L]} > 1-\e$ for some $\e > 0$ and $L \geq l_0$.
	Then there exist $\alpha,\beta \in \RR$ such that 
		\begin{equation}\label{eq:146:01}
	f(n) = e(n \alpha + \beta + O(h(\e))) \quad \text{ for $h(\e)$-almost all $n$ with $n < q^L$ and $q^{l_0} \mid n$}.
	\end{equation}
\end{corollary}
\begin{proof} Note that we may assume that $\e$ is arbitrarily small in absolute terms, else \eqref{eq:146:01} is vacuously true for a suitable choice of $h$. 
It follows from the ``99\% variant'' of the Inverse Theorem for Gowers norms (Theorem \ref{thm:inverse-99}) that there exists a polynomial $p \in \RR[x]$ with $\deg p < s$ such that 
\begin{equation}\label{eq:36:00}
	f(n) = e( p(n) + O(\delta) ) \text{ for $\delta$-almost all } n < q^L,
\end{equation} 
where $\delta = \delta(\e) \to 0$ as $\e \to 0$. By Proposition \ref{prop:rigidity} we have 
\begin{equation}\label{eq:36:01}
	f(n) = e( \alpha n + b(n) + O(\delta) ) \text{ for $\delta$-almost all } n < q^L,
\end{equation} 
where $\alpha \in \RR$ and $b(n) = 0$ for $n$ with $q^{l_0} \mid n$, provided that that $l_0$ is sufficiently large and $\e$ is sufficiently small. \end{proof}
\begin{remark}\label{rmrk:pretentious=>linear}
	The above Corollary \ref{cor:pretentious=>linear} can also be construed as the ``99\% variant'' of Theorem \ref{thm:oscillating=>Gowers-uniform}. 
	For future reference, we note that in said corollary, if $\e$ is sufficiently small in terms of $L$, then \eqref{eq:146:01} holds for all $n$ in the prescribed range (rather than almost all). This can be construed as a (slight strengthening of) the ``100\% variant'' of Theorem \ref{thm:oscillating=>Gowers-uniform}.
\end{remark}

\makeatletter{}\section{The end of the chase}

We are now ready to finish the proof of Theorem \ref{thm:oscillating=>Gowers-uniform}. 
The main missing ingredient is contained in the following proposition. The proof of Theorem \ref{thm:oscillating=>Gowers-uniform} is immediate by combining the following Proposition \ref{prop:oscillating=>E-saving} and Lemma \ref{lem:wlog-N=q^L}.

\begin{proposition}\label{prop:oscillating=>E-saving}
There exists a constant $\kappa > 0$ such that for any $q$-multiplicative sequence $f \colon \NN_0 \to \UU$ and any $L \geq 0$ it holds that
	\begin{equation}\label{eq:47:90}
		\norm{f}_{U^s[q^L]} \ll \norm{f}_{U^2[q^L]}^{\kappa}.
	\end{equation}
\end{proposition}
\begin{proof}
	Our strategy is to construct a partition $\NN_0 = \bigcup_{i = 0}^{\infty} I_i$ of the positive integers into a sequence of intervals $I_i = [K_i,K_{i} + L_i)$ such that for any $m \geq 0$ and any integer $K \in I_m$ we have a lower bound 	
	\begin{equation}\label{eq:47:00}
		\norm{f}_{U^2[q^K]} \geq \sup_{\alpha \in \RR} \abs{ \EEE_{n< q^K} f(n)e(-n\alpha) } \gg \exp( - C m),
	\end{equation}
	as well as an upper bound on the Gowers norms
	\begin{equation}\label{eq:47:01}
		\norm{f}_{U^s[q^K]} \ll \exp(-c m)
	\end{equation}
	for some constants $C,c > 0$. Then \eqref{eq:47:90} will follow \eqref{eq:47:00} and \eqref{eq:47:01} with $\kappa = c/C$. 
	
	Minor technical complications arise when $\norm{f}_{U^2[q^L]}$ fails to tend to $0$ as $L \to \infty$. In this case instead of an infinite partition of $\NN_0$ we could consider a finite one $\NN_0 = \bigcup_{i=0}^{i_{\max}}I_i$. While the argument can be adapted almost verbatim to that situation, in order to simplify the notation we proceed differently. Note that \eqref{eq:47:90} does not depend on the values $f(n)$ with $n \geq q^L$. Hence, we may  replace $f$ by a $q$-multiplicative function such that $f(aq^l) = -1$ for $1 \leq a < q$ and $l \geq L$ (it follows from \cite{Gelfond-1968} that the resulting function has $\norm{f}_{U^2[q^L]} \to 0$ as $L \to \infty$; here, we crucially use the fact that the implicit constant in \eqref{eq:47:90} does not depend on $f$). 
	
	Let $\delta > 0$ be a small parameter to be specified in the course of the proof, let $l_0$ be the constants appearing in Corollary \ref{cor:pretentious=>linear}. 	
	 We put $K_0 = 0$ and once $K_i$ has been defined we choose $L_i$ to be the least integer such that there exists no $\alpha,\beta \in \RR$ such that
	\begin{equation}\label{eq:20:01}
	S^{K_i}f(n) = e(n q^{K_i}\alpha + \beta +  \OO(\delta)) \quad \text{ for $\delta$-almost all $n$ with $n < q^{L_i}$ and $q^{l_0} \mid n$}.
	\end{equation}
By construction, $L_i > l_0$, else \eqref{eq:20:01} is vacuously true for $\alpha = 0$. Also, $L_i$ is well defined because $\norm{S^{K_i}f}_{U^2[q^L]} \to 0$ as $L \to \infty$. Next, we put $K_{i+1} = K_i + L_i$. 

We consider the upper bound \eqref{eq:47:01} first since the definitions are set up so that this step is simpler. It follows from Corollary \ref{cor:pretentious=>linear} and the definition of the sequence $L_i$ that there exists a parameter $\e = \e(\delta) > 0$ such that
	\begin{equation}\label{eq:146:00a}
	\norm{S^{K_i}f}_{U^s[q^{L_i}]} < 1 - \e
	\end{equation}
for all $i \in \NN_0$.
Hence, Proposition \ref{prop:E-bound=>A-bound} implies that there is a constant $c = c(\e) > 0$ such that \eqref{eq:47:01} holds for all $K \geq K_m$.

To obtain the lower bound \eqref{eq:47:00} we further refine the partition $\NN_0 = \bigcup_{i=0}^\infty I_i$ by setting 
\begin{align}
K_i' &:= K_i+l_0, \qquad L_i' := L_i - l_0 -1 \geq 0, \label{eq:29:01} \\
I_i' &:= [K_i',K_i' + L_i') = [K_i+l_0, K_{i+1} - 1) \label{eq:29:02}
\\ J_i &:= [K_i'-l_0-1,K_i') = [K_i-1,K_i+l_0). \label{eq:29:03} 
\end{align} 
for $i \geq 1$. Additionally, for $i = 0$ we use the same definitions \eqref{eq:29:01} and \eqref{eq:29:02} and instead of \eqref{eq:29:03} we put $J_0 = [0,K_0')$. This is set up so that $\NN_0 = \bigcup_{i = 0}^{\infty} J_i \cup I_i'$, the lengths of $J_i$ are bounded (indeed, they all have lengths $l_0+1$ or $l_0$) and for each $i \geq 0$ it follows from how $L_i$ are defined that there exist $\alpha_i,\beta_i \in \RR$ such that 
	\begin{equation}\label{eq:20:01a}
	S^{K_i'}f(n) = e(n q^{K_i'}\alpha_i + \beta_i + \OO(\delta)) \quad \text{ for $\delta$-almost all $n < q^{L_i'}$}.
	\end{equation}
(Note that existence of such $\alpha_i$ is trivial if $L_i' = 0$.)

Fix an integer $M \geq 0$. Our strategy is to combine the information encoded in $\alpha_i$ for all intervals $I_i$ to construct $\alpha \in \RR$ satisfying \eqref{eq:47:00} for all $m < M$ and $K \in I_m$. 
Recall that 
\begin{equation}\label{eq:47:02}
	\abs{ \EEE_{n<q^K} f(n)e(-n\alpha) } = \prod_{l < K} \abs{ \EEE_{a < q} f(a q^l) e(- a q^l \alpha) }.
\end{equation}
Hence, it will suffice to ensure the following two conditions hold:
\begin{align}
\label{eq:28:01}	\prod_{l \in J_i} \abs{ \EEE_{a < q} f(a q^l) e(- a q^l \alpha) } = \exp( - O(1)),\\
\label{eq:28:00}	\prod_{l \in I_i'} \abs{ \EEE_{a < q} f(a q^l) e(- a q^l \alpha) } = \exp( - O(1)).
\end{align}

We shall construct $\alpha = 0.\alpha^{(0)}\alpha^{(1)}\alpha^{(2)}\dots$ digit by digit and show that it satisfies \eqref{eq:28:01} and \eqref{eq:28:00}. The digits appearing at positions exceeding $K_M$ will not play a significant role, and for concreteness we set $\alpha^{(l)} = 0$ for $l \geq K_M-1$. For $l \in I_i'$ with $i < M$, in the regime relevant for \eqref{eq:28:00} $f$ behaves like $e(n\alpha_i)$, which motivates us to set $\alpha^{(l)} = \alpha_{i}^{(l)}$ (strictly speaking, for technical reasons we will need to modify this slightly, but for the sake of clarity we leave the definition as it is and explain the modification later). We construct the remaining digits in a descending order. If $l \in J_i$ with $i < M$ and $\alpha^{(k)}$ have been constructed for $k > l$ then it follows from  Parseval's identity that we may choose $\alpha^{(l)}$ such that
\begin{equation}\label{eq:47:10}
\abs{ \EEE_{a < q} f(a q^l) e(- a q^l \alpha) } \geq 
\bra{\EEE_{b<q} \abs{ \EEE_{a < q} f(a q^l) e\big(- a (q^l \alpha +b/q) \big) }^2 }^{1/2}
 = q^{-1/2} =  \exp(-O(1)).
\end{equation}
(Note that the expressions above are independent of the digits of $\alpha$ at positions $<l$.)
Since the intervals $J_i$ have bounded lengths, \eqref{eq:28:01} follows directly from \eqref{eq:47:10} (in fact, the left hand side of \eqref{eq:28:01} is $\geq q^{-(l_0+1)/2}$). It remains to verify \eqref{eq:28:00}.

Fix $i < M$ and put $\gamma = q^{K_i'}(\alpha - \alpha_i)$. If follows from \eqref{eq:20:01a} that
\begin{align}
\label{eq:21:01} \abs{ \prod_{l \in I_i'} \bra{ \EEE_{a < q} f(a q^{l} ) e(- a q^{l} \alpha) } }
 &= \abs{ \EEE_{n < q^{L'_i}} f(nq^{K_i'})e(-n q^{K_i'} \alpha) }
\\  &= \abs{ \EEE_{n < q^{L_i'}} e(-n \gamma)} + O(\delta) 
\\ &= \abs{ \frac{1 - e(q^{L_i'} \gamma)}{q^{L_i'} (1-e(\gamma))}} +  O(\delta)\label{eq:21:02} =
\varphi_{L_i'}(q^{L_i'} \gamma) + O(\delta),
\end{align}
where $\varphi_L(x) := (1-e(x))/(q^L(1-e(x/q^L))$.
Since the digits of $\alpha$ and $\alpha_i$ agree on positions between $K_i'$ and $K_i'+L_i$, it follows that $\abs{\gamma} \leq q^{-{L_i'}}$. We note that $\varphi_L(0) = 1$, $\varphi_L(\pm 1) = 0$ and $\varphi_L$ has no other zeros in $[-1,1]$ for $L \geq 1$ and $\varphi_L(x) \to \varphi_{\infty}(x) := \abs{1-e(x)}/(2\pi x)$ as $L \to \infty$ uniformly in $x$.

Now, \eqref{eq:28:01} follows directly from \eqref{eq:21:02}, unless the main term in \eqref{eq:21:02} is small, which is the case when $\abs{\gamma}$ is close to $q^{-{L_i'}}$. More precisely, let us pick a parameter $\delta' > 0$, small but considerably larger than $\delta$; we aim to show that the left hand side of \eqref{eq:28:00} is $\gg \delta'$. Assuming that $\delta'/\delta$ is sufficiently large (in absolute terms), the expression in \eqref{eq:28:00} is $\geq \delta'/2$ provided that 
\begin{equation}\label{eq:21:50}
\varphi_{L_i'}(q^{L_i'} \gamma) \geq \delta',
\end{equation}
in which case we are done.

Suppose now that \eqref{eq:21:50} fails to hold. A standard calculus argument shows that
$q^{L_i'} \gamma = \pm 1+O(\delta')$, meaning that the expansion of $\gamma$ contains a long string of $(q-1)$'s starting at position $L_i'$ (of length $\geq \log_q 1/\delta' - O(1)$). In this case, we need to go back in the construction and replace $\alpha_i$ with $\alpha_i' = \alpha_i \pm O( q^{-L_i'}\delta')$ and change the relevant digits of $\alpha$ accordingly, where $\alpha_i'$ is chosen in such a way that for the new value of $\gamma$, say $\gamma'$, we have $q^{L_i'}\abs{\gamma'} \ll  \delta'$. This does not affect any of the estimates in a substantial way since the analogue of \eqref{eq:20:01a} remains valid, except that in several places $\delta$ needs to be replaced with $\delta'$.

It now follows by elementary analysis that if $\delta'$ is chosen sufficiently small then
\begin{equation}\label{eq:21:50b}
\varphi_{L_i'}(q^{L_i'} \gamma') \geq 1/2.
\end{equation}
Hence, the left hand side of (the modified version of) \eqref{eq:28:00} is $\geq 1/2 = \exp(-O(1))$. This concludes the proof of \eqref{eq:28:00}, and hence also of \eqref{eq:47:00}.
\end{proof}

\makeatletter{}\section{Applications and further directions}\label{sec:Epilogue}

\subsection*{Strongly multiplicativity}
A sequence $f \colon \NN_0 \to \UU$ is {\em strongly $q$-multiplicative} if it is multiplicative and additionally $f(qn)= f(n)$ for all $n$. Such a sequence is uniquely determined by the values $f(1), \dots, f(q-1)$. For instance, the sequence $f(n) = e( s_q(n) \alpha)$ is strongly $q$-multiplicative for every $\alpha \in \RR$. For these sequences, the assumptions of Theorem \ref{thm:oscillating=>Gowers-uniform} are particularly easy to check.

\begin{proposition}\label{prop:strong-multi-dichotomy}
	Suppose that $f \colon \NN_0 \to \UU$ is strongly $q$-multiplicative. Then $f$ is of Gelfond type of order $1$ unless there exists  $0 \leq p < q-1$ such that $f(n) = e( n p/(q-1))$ for all $n$.
\end{proposition}
\begin{proof}
	Let $\phi(\alpha) = \abs{ \EE_{a < q} f(a) e(-a \alpha)}$. Since
	$$
		\abs{ \EEE_{n<q^L} f(n) e(- \alpha n) } = \prod_{l < L} \phi(q^l \alpha),
	$$ 
	it is easy to check that $f$ is of Gelfond type unless $\sup_\beta \phi(\beta) \phi(q\beta) = 1$. Suppose that the latter condition holds. Since $\phi$ is clearly continuous, there exists $\beta$ with $\phi(\beta) = \phi(q \beta) = 1$. By the triangle inequality, this in turn implies that
	$$ 1 = f(0) = f(a) e(- a \beta) = f(a) e(- a q \beta) $$
	for any $1 \leq a < q$. In particular, $f(a) = e(a\beta)$ for $0 \leq a < q$ and $(q-1) \beta \in \ZZ$. It follows that $f$ is given by $f(n) = e(n \beta)$, hence it falls into one of the cases mentioned above. 
\end{proof}

In other words, any strongly $q$-multiplicative sequence is of Gelfond type of order $1$ (and hence Gowers uniform of all orders), unless it is periodic (in which case, it has period $q-1$). 

\subsection*{Counting patterns}
We are now ready to outline the proof of Theorem \ref{thm:sum-of-digits-patterns}. Because this is a standard application, we skip some of the technical details. Note that by a slightly more careful reasoning it is possible to obtain a similar result with a more explicit error term.

\begin{proof}[\ifLMS{}{Proof}of Theorem \ref{thm:sum-of-digits-patterns}]
	We first prove the second part of the statement. Fix an integer $Q$ coprime to $q-1$. By Theorem \ref{thm:oscillating=>Gowers-uniform} and Proposition \ref{prop:strong-multi-dichotomy}, all sequences $e(s_q(n)p/Q)$ are Gowers uniform of all orders for $0 < p < Q$. Put 
	$$A_j = \set{ n \in \NN_0 }{ s_q(n) \equiv j \bmod Q }.$$ Applying Fourier decomposition
	$$1_{A_j}(n) = \EEE_{a < Q} e\bra{ \frac{a}{Q}( s_q(n)-j)},$$
	we conclude that the balanced characteristic functions $1_{A_j} - {1}/{Q}$ are Gowers uniform of all orders. It follows from the generalised von Neumann theorem, Theorem \ref{thm:von-Neumann}, that there are $\gg N^2$ progressions $n,n+m,\dots,n+(k-1)m$ contained in $[N]$ with $n+j m \in A_{r_j}$ for $0 \leq j < k$.
	
	Coming back to the first part of the statement, approximating $1_{I_j}$ by trigonometric polynomials and repeating the previous argument, we see that it will suffice to show that for each $a \in \ZZ \setminus \{0\}$, the sequence $e( a s_{q}(n) \alpha)$ is Gowers uniform of all orders. This follows directly from Theorem \ref{thm:oscillating=>Gowers-uniform}.
\end{proof}

\subsection*{Bertrandias pseudorandomness}
A function $f \colon \NN_0 \to \UU$ is \emph{pseudorandom} in the sense of Bertrandias \cite{Bertrandias-1964}, (see also \cite{Coquet-1976,CoquetKamaeMendesFrance-1977}) if the correlation coefficients
\begin{equation}
\label{eq:60:00}
	\gamma_r = \lim_{N \to \infty} \EEE_{n<N} f(n+r) \bar f(n) 
\end{equation}
exist for all $r \geq 0$ and converge to $0$ in density, i.e.,
\begin{equation}
\label{eq:60:01}
	\lim_{R \to \infty} \EEE_{r < R} \abs{\gamma_r}^2 = 0,
\end{equation}
meaning that $f$ admits a spectral measure which is continuous (without atoms). 
It follows from \cite{Coquet-1976} that a $q$-multiplicative sequence is pseudorandom in the sense of Bertrandias if and only if for each $\alpha \in \RR$, $\EE_{n<N} f(n) e(n \alpha) \to 0$ as $N \to \infty$, i.e., $f$ is oscillating of order $1$. Related results are also obtained by Spiegelhofer for the Ostrowski sums of digits function \cite{Spiegelhofer-2016}.

If $f$ additionally is of Gelfond type, then we can prove a sharper estimate, namely
\begin{equation}
\label{eq:60:90}
 \EEE_{r < R} \abs{\gamma_r}^2 \ll R^{-c}
\end{equation}
for a constant $c > 0$. This implies that the Hausdorff dimension of the spectral measure is strictly positive. We first show that the limits defining $\gamma_r$ converge rapidly:
\begin{equation}
\label{eq:60:91}
\EEE_{n<N} f(n+r) \bar f(n) = \gamma_r + O(r\log N/N).
\end{equation}
For $r = 1$, considering the largest power of $l$ dividing $n$ we obtain a simple estimate:
\begin{align*}
\EEE_{n<N} f(n+1) \bar f(n) &= \sum_{l=0}^\infty \sum_{a=1}^q \PP_{n < N}\bra{n+1 \equiv aq^l\bmod{q^{l+1}}} f(aq^{l})\bar f(aq^{l}-1) 
\\&= \sum_{l=0}^{\infty} \sum_{a=1}^q \bra{ q^{-(l+1)} f(aq^{l})\bar f(aq^{l}-1) + O(\min(1/N,1/q^l)) } 
\\&= \gamma_r + O\bra{{\log N}/{N}}.
\end{align*}
For $r > 1$ one can either run an analogous computation or block the terms in \eqref{eq:60:91} into intervals whose length is the least power of $q$ exceeding $r$.
Letting $R$ be a large integer and taking $N = R^{1+\e}$ where $\e > 0$ is a small constant we have
\begin{equation}
\label{eq:60:02}
	\EEE_{r < R} \abs{\gamma_r}^2 = R^{\e} \EEE_{r,n,m < N} f(n+r) \bar f(n) \bar f(m + r) \bar f(m) 1_{[R]}(r) + O(\log R/R^{\e}).
\end{equation}
The claim \eqref{eq:60:90} now follows by estimating the right hand side of \eqref{eq:60:02} using the generalised von Neumann theorem \ref{thm:von-Neumann} and applying Theorem \ref{thm:oscillating=>Gowers-uniform}.

\subsection*{Quasimultiplicativity} In \cite{KropfWagner-2017}, Kropf and Wagner introduce the notion of a $q$-quasimultiplicative sequence, which is considerably weaker than that of a strongly $q$-multiplicative sequence (by an unfortunate quirk of the terminology, a $q$-multiplicative sequence need not be $q$-quasimultiplicative\footnote{In general, it might be more apt to use the term ``strongly $q$-quasimultiplicative''; however, here we keep the existing terminology for the sake of consistency, cf.\ \cite{Konieczny-2018-semimulti}.}). 

A sequence $f \colon \NN_0 \to \UU$ is {\em  $q$-quasimultiplicative} if there exists $r \geq 0$ such that $f(m + q^{t+r} n) = f(m)f(n)$ for any $n,m \in \NN_0$ with $m < q^t$. This class turns out to contain many interesting examples (see \cite{KropfWagner-2017} for details), especially the block counting sequences. In particular, the Rudin--Shapiro sequence (taking values $\pm 1$ depending on the parity of the number of appearances of the block $11$ in binary expansion of $n$) is $2$-quasimultiplicative but not $2$-multiplicative.

In light of the fact that the Rudin--Shapiro sequence is known to be Gowers uniform of all orders by \cite{Konieczny-2017+}, it seems plausible that our results can be extended to $q$-quasimultiplicative sequences. We do not pursue this question further in this paper, except to state the following conjecture.

\begin{conjecture}\label{conj:quasimultiplicative}
	Let $f \colon \NN_0 \to \UU$ be a $q$-quasimultiplicative sequence
	such that
	\begin{equation}\label{eq:osc-301}
		\sup_{ \alpha \in \RR } \abs{ \EEE_{n < N} f(n) e( \alpha n )} \ll N^{-c}. 
	\end{equation}
	Then for any $s \geq 1$ there exists $c_s > 0$ such that 
	\begin{equation}\label{eq:osc-302}
		\norm{f}_{U^s[N]} \ll N^{-c_s}.
	\end{equation}
\end{conjecture}

\subsection*{Boundedness}
Throughout the paper, we work with sequences $f \colon \NN_0 \to \CC$ taking values in $\UU$. With minor modifications, this could be weakened to the assumption that $\abs{f(n)} \leq 1$ for all $n \in \NN_0$. It is possible that similar results could be obtained for sufficiently slowly growing sequences. However, we do not pursue this line of inquiry further. 

\makeatletter{} \appendix
\section{Alternative proof of Theorem \ref{thm:oscillating=>fully-oscillating}}

If one is only interested in proving Theorem \ref{thm:oscillating=>fully-oscillating}, an arguably simpler argument is possible. We shall need a preliminary lemma.
We note in passing that this lemma together with Theorem 3 in \cite{Fan2017c} yields a slightly stronger conclusion in Theorem \ref{thm:WET}.
 
\begin{lemma}\label{prop:oscillating-equivalent}
	Suppose that $f \colon \NN_0 \to \UU$ is $q$-multiplicative and $c > 0$. Then the estimate 
		\begin{equation}\label{eq:osc-12}
		\sup_{\substack{ p \in \RR[x] \\ \deg p \leq d}} \abs{ \EEE_{n < q^L} f(n) e(p(n))} \ll q^{-cL}
	\end{equation}
	is equivalent to the apparently stronger estimate 
	\begin{equation}\label{eq:osc-13}
		\sup_{M \geq 0} \sup_{\substack{ p \in \RR[x] \\ \deg p \leq d}} \abs{ \EEE_{n < N} f(n+M) e(p(n))} \ll N^{-c}.
	\end{equation}
\end{lemma}
\begin{proof}
	Assume that \eqref{eq:osc-12} holds. Because $f$ is $q$-multiplicative, for any $t, m \geq 0$ with $q^t \mid m$ we have
\begin{equation}\label{eq:osc-14}	
		\sup_{\substack{ p \in \RR[x] \\ \deg p \leq d}}  \abs{ \EEE_{n < q^t} f(n+m) e(p(n))}
		= \sup_{\substack{ p \in \RR[x] \\ \deg p \leq d}}  \abs{ \EEE_{n < q^t} f(n) e(p(n))} \ll q^{-ct}.
	\end{equation}
	It remains to split the interval $[M,M+N)$ into $q$-adic intervals of the form $[m_j, m_j+q^{t_j})$ where each $t_j$ appears $O(1)$ times, and to bound each of thus obtained sums independently, using \eqref{eq:osc-14}. 
\end{proof}

We now proceed to the proof of Theorem \ref{thm:oscillating=>fully-oscillating}. Since the argument is presented here only for expository purposes, we restrict our attention to the special case of monomials of degree $2$. The argument is closely modelled on the proof of a similar assertion in \cite{EisnerKonieczny-2018}.

\begin{proof}[\ifLMS{}{Proof }of Theorem \ref{thm:oscillating=>fully-oscillating}, case $p(n) = \alpha n^2$]
	We shall show that under the assumption 
	\begin{equation}\label{eq:osc-01x}
		\sup_{ \alpha \in \RR } \abs{ \EEE_{n < N} f(n) e( \alpha n )} \ll N^{-c_1} 
	\end{equation}
	 we have the bound	
	\begin{equation}\label{eq:osc-03}
		\sup_{\alpha \in \RR} \abs{ \EEE_{n < N} f(n) e(\alpha n^2)} \ll N^{-c_2}, 
	\end{equation}
	where $c_2 > 0$ is a constant, dependent only on $c_1$, to be determined in the course of the proof. By Lemma \ref{prop:oscillating-equivalent}, we may assume that $N = q^L$ is a power of $q$.
	
    Fix a value of $\alpha$, and denote the expression on the left hand side of \eqref{eq:osc-03} by $S :=\abs{ \EE_{n < N} f(n) e(\alpha n^2)} $. The strategy of the proof is to attempt to show \eqref{eq:osc-03} by the van der Corput inequality. Either this attempt succeeds, in which case we are done, or we are able to extract useful information about $\alpha$, which then enables us to apply a different method.
 	
	Let $\delta > 0$ be a small constant, and let $H \sim q^{\delta L}$ be a power of $q$. Applying van der Corput inequality, we conclude that 
	\begin{equation}\label{eq:osc-04}
		 S^2 \ll \EEE_{h < H} \abs{ \EEE_{n < N} \bar f(n) f(n+h) e(2 h \alpha n)} + O(H/N).
	\end{equation}
	As the next step, we further split the inner average in \eqref{eq:osc-04}. Let $Q \sim q^{2\delta L}$ be a power of $q$; writing any $n < N$ as $n = Q n' + m$ with $m < Q$ we obtain
	\begin{equation}\label{eq:osc-05}
		 S^2 \ll \EEE_{h < H} \EEE_{m < Q} \abs{ \EEE_{n < N/Q} \bar f(Q n + m) f(Qn+m+h) e(2 h Q \alpha n)} + O(H/N).
	\end{equation}
	The inner average in \eqref{eq:osc-05} is particularly easy to analyse if $ m+h < Q$, in which case $q$-multiplicativity of $f$ implies that $\bar f(Q n + m) f(Qn+m+h) = \bar f(m) f(m+h)$, whence we can simplify:
	\[
	\abs{ \EEE_{n < N/Q} \bar f(Q n + m) f(Qn+m+h) e(2 h Q \alpha n) } =
	\abs{ \EEE_{n < N/Q} e(2 h Q \alpha n)}.
	\]
	Using the trivial bound when $m+h \geq Q$ or $h = 0$, and applying the classical estimate $\abs{ \EE_{n < X} e(\theta)} \ll 1/(X \norm \theta)$, we conclude that 
	\begin{equation}\label{eq:osc-06}
	S^2 \ll \frac{1}{N/Q} \cdot \frac{1}{ \min_{1 \leq h < H} \norm{2hQ \alpha} } + O(1/H) + O(H/Q).
	\end{equation}

If $S^2 \ll q^{-\delta L}$, we are done (and we may take $c_2 = \delta/2$). Otherwise, \eqref{eq:osc-06} implies that there exists $k \ll q^{3 \delta L}$ with $k \neq 0$, such that $\norm{ k \alpha } \ll N^{-1 + 3 \delta}$. Write $\alpha = l/k + \zeta$ with $l \in \ZZ$ and $\abs{\zeta} \ll N^{-1 + 3 \delta}$. 
	
We may now use the above Diophantine approximation of $\alpha$ to bound the portion of the sum defining $S$ corresponding to a long interval. It will be convenient to use the Fourier expansion 
$$
	e(n^2 j/k ) = \sum_{i = 0}^{k-1} a_{i,j} e(n i/k),
$$
where $\abs{a_{i,j}} \leq 1$. 

Pick $M \sim q^{L/10}$, a power of $q$. For any $m_0 \geq 0$ with $M \mid m_0$ we may estimate that
	\begin{align*}
	\abs{ \EEE_{n \in [m_0, m_0 + M) } f(n) e\bra{ \alpha n^2 } } &=
		\abs{ \EEE_{n < M} f(n) e\bra{ 2 \alpha m_0 n + n^2 \alpha } }
		\\ & = \abs{ \sum_{i=0}^{k-1} \EEE_{n < M} f(n) a_{i,l} e\bra{ (2 \alpha m_0 + i/k)n } } + O(M^2 \norm{ \zeta})
		\\ & = O(k M^{-c_1}) + O(M^2 \norm{\zeta}) = O(N^{-c_2}),
	\end{align*}
	where in the last line $c_1$ is the constant from the assumption \eqref{eq:osc-01x}, and 
	\begin{equation}\label{eq:def-of-c2}
		c_2 = \min\bra{ c_1/10 - 3 \delta,\ 8/10 - 3 \delta};
	\end{equation}
	we assume $\delta$ was chosen small enough that $c_2 > 0$. It follows by splitting $[0,N)$ into disjoint intervals $[m_i,m_i+M)$ that $S \ll N^{-c_2}$, as needed.
\end{proof}

\makeatletter{} 
\makeatletter{}

\bibliographystyle{alphaabbr}
\bibliography{bibliography}

\newcommand{\etalchar}[1]{$^{#1}$}
\begin{thebibliography}{AKK{\etalchar{+}}03}

\bibitem[AJ17]{Akiyama-Jiang}
S.~Akiyama and Y.~Jiang.
\newblock {Higher order oscillation and uniform distribution}.
\newblock 2017.
\newblock Preprint.
  {\href{https://arxiv.org/abs/1710.08643}{arXiv:1612.08376v1. [math.DS]}}.

\bibitem[AKK{\etalchar{+}}03]{AlonKaufmanKrivelevichLitsynRon-2003}
N.~Alon, T.~Kaufman, M.~Krivelevich, S.~Litsyn, and D.~Ron.
\newblock Testing low-degree polynomials over {${\rm GF}(2)$}.
\newblock In {\em Approximation, randomization, and combinatorial
  optimization}, volume 2764 of {\em Lecture Notes in Comput. Sci.}, pages
  188--199. Springer, Berlin, 2003.

\bibitem[Ber64]{Bertrandias-1964}
J.-P. Bertrandias.
\newblock Suites pseudo-al\'eatoires et crit\`eres d'\'equir\'epartition modulo
  un.
\newblock {\em Compositio Math.}, 16:23--28 (1964), 1964.

\bibitem[BTZ10]{BergelsonTaoZiegler-2010}
V.~Bergelson, T.~Tao, and T.~Ziegler.
\newblock An inverse theorem for the uniformity seminorms associated with the
  action of {$\Bbb F^\infty_p$}.
\newblock {\em Geom. Funct. Anal.}, 19(6):1539--1596, 2010.

\bibitem[CKMF77]{CoquetKamaeMendesFrance-1977}
J.~Coquet, T.~Kamae, and M.~Mend\`es~France.
\newblock Sur la mesure spectrale de certaines suites arithm\'etiques.
\newblock {\em Bull. Soc. Math. France}, 105(4):369--384, 1977.

\bibitem[Coq75]{Coquet-1975}
J.~Coquet.
\newblock Sur les fonctions {$q$}-multiplicatives presque-p\'eriodiques.
\newblock {\em C. R. Acad. Sci. Paris S\'er. A-B}, 281(2-3):Ai, A63--A65, 1975.

\bibitem[Coq76]{Coquet-1976}
J.~Coquet.
\newblock Sur les fonctions {$q$}-multiplicatives pseudo-al\'eatoires.
\newblock {\em C. R. Acad. Sci. Paris S\'er. A-B}, 282(4):Ai, A175--A178, 1976.

\bibitem[Coq77]{Coquet-1977}
J.~Coquet.
\newblock Fonctions {$q$}-multiplicatives. {A}pplication aux nombres de
  {P}isot-{V}ijayaraghavan.
\newblock In {\em S\'eminaire de {T}h\'eorie des {N}ombres (1976--1977)}, pages
  Exp. No. 17, 15. CNRS, Talence, 1977.

\bibitem[Coq80]{Coquet-1980}
J.~Coquet.
\newblock R\'epartition modulo {$1$}\ des suites {$q$}-additives.
\newblock {\em Comment. Math. Prace Mat.}, 21(1):23--42, 1980.

\bibitem[Del72]{Delange-1972}
H.~Delange.
\newblock Sur les fonctions {$q$}-additives ou {$q$} -multiplicatives.
\newblock {\em Acta Arith.}, 21:285--298. (errata insert), 1972.

\bibitem[DL01]{Drmota-2001}
M.~Drmota and G.~Larcher.
\newblock The sum-of-digits-function and uniform distribution modulo 1.
\newblock {\em J. Number Theory}, 89(1):65--96, 2001.

\bibitem[DT06]{DartygeTenenbaum-2006}
C.~Dartyge and G.~Tenenbaum.
\newblock Congruences de sommes de chiffres de valeurs polynomiales.
\newblock {\em Bull. London Math. Soc.}, 38(1):61--69, 2006.

\bibitem[EA17]{Elabdalaoui}
E.~H. El~Abdalaoui.
\newblock {Oscillating sequences, Gowers norms and Sarnak’s conjecture}.
\newblock 2017.
\newblock Preprint. {\href{https://arxiv.org/pdf/1704.07243.pdf}{
  arXiv:1704.07243 [math.DS]}}.

\bibitem[EK18]{EisnerKonieczny-2018}
T.~Eisner and J.~Konieczny.
\newblock Automatic sequences as good weights for ergodic theorems.
\newblock {\em Discrete Contin. Dynam. Systems}, 2018.
\newblock (to be published).

\bibitem[ET12]{EisnerTao-2012}
T.~Eisner and T.~Tao.
\newblock Large values of the {G}owers-{H}ost-{K}ra seminorms.
\newblock {\em J. Anal. Math.}, 117:133--186, 2012.

\bibitem[Fan16]{Fan2016}
A.~Fan.
\newblock {Oscillating sequences of higher orders and topological systems of
  quasi-discrte spectrum}.
\newblock 2016.
\newblock Preprint. {\href{https://arxiv.org/abs/1802.05204}{
  arXiv:1802.05204[math.DS]}}.

\bibitem[Fan17a]{Fan2017}
A.~Fan.
\newblock Fully oscillating sequences and weighted multiple ergodic limit.
\newblock {\em C. R. Math. Acad. Sci. Paris}, 355(8):866--870, 2017.

\bibitem[Fan17b]{Fan2017c}
A.~Fan.
\newblock {Weighted Birkhoff ergodic theorem with oscillating weights}.
\newblock {\em Ergodic Theory Dynam. Systems}, 2017.
\newblock (to appear).

\bibitem[Fan18]{Fan2017b}
A.-h. Fan.
\newblock Topological {W}iener-{W}intner ergodic theorem with polynomial
  weights.
\newblock {\em Chaos Solitons Fractals}, 117:105--116, 2018.

\bibitem[FJ18]{FanJiang}
A.-H. Fan and Y.~Jiang.
\newblock Oscillating sequences, {MMA} and {MMLS} flows and {S}arnak's
  conjecture.
\newblock {\em Ergodic Theory Dynam. Systems}, 38(5):1709--1744, 2018.

\bibitem[FM96a]{FouvryMauduit-1996-AA}
E.~Fouvry and C.~Mauduit.
\newblock M\'ethodes de crible et fonctions sommes des chiffres.
\newblock {\em Acta Arith.}, 77(4):339--351, 1996.

\bibitem[FM96b]{FouvryMauduit-1996-MA}
E.~Fouvry and C.~Mauduit.
\newblock Sommes des chiffres et nombres presque premiers.
\newblock {\em Math. Ann.}, 305(3):571--599, 1996.

\bibitem[FM96c]{FM1996a}
E.~Fouvry and C.~Mauduit.
\newblock Sommes des chiffres et nombres presque premiers.
\newblock {\em Math. Ann.}, 305(3):571--599, 1996.

\bibitem[FM05]{FouvryMauduit-2005}
E.~Fouvry and C.~Mauduit.
\newblock Sur les entiers dont la somme des chiffres est moyenne.
\newblock {\em J. Number Theory}, 114(1):135--152, 2005.

\bibitem[FSS18]{FSS2018}
A.~Fan, J.~Schmeling, and W.~Shen.
\newblock {Asymptotic behaviours of generalized Thue-Morse trigonometric
  polynomials}.
\newblock 2018.
\newblock Preprint.

\bibitem[Gel68]{Gelfond-1968}
A.~O. Gel'fond.
\newblock Sur les nombres qui ont des propri\'et\'es additives et
  multiplicatives donn\'ees.
\newblock {\em Acta Arith.}, 13:259--265, 1967/1968.

\bibitem[Gow01]{Gowers-2001}
W.~T. Gowers.
\newblock A new proof of {S}zemer\'edi's theorem.
\newblock {\em Geom. Funct. Anal.}, 11(3):465--588, 2001.

\bibitem[Gra09]{Granville-2009}
A.~Granville.
\newblock Pretentiousness in analytic number theory.
\newblock {\em J. Th\'eor. Nombres Bordeaux}, 21(1):159--173, 2009.

\bibitem[Gre]{Green-book}
B.~Green.
\newblock {\em {Higher-Order Fourier Analysis, I}}.
\newblock (Notes available from the author).

\bibitem[GT08]{GreenTao-2008}
B.~Green and T.~Tao.
\newblock An inverse theorem for the {G}owers {$U^3(G)$} norm.
\newblock {\em Proc. Edinb. Math. Soc. (2)}, 51(1):73--153, 2008.

\bibitem[GT10]{GreenTao-2010}
B.~Green and T.~Tao.
\newblock Linear equations in primes.
\newblock {\em Ann. of Math. (2)}, 171(3):1753--1850, 2010.

\bibitem[GTZ12]{GreenTaoZiegler-2012}
B.~Green, T.~Tao, and T.~Ziegler.
\newblock An inverse theorem for the {G}owers {$U^{s+1}[N]$}-norm.
\newblock {\em Ann. of Math. (2)}, 176(2):1231--1372, 2012.

\bibitem[GW10]{GowersWolf-2010}
W.~T. Gowers and J.~Wolf.
\newblock The true complexity of a system of linear equations.
\newblock {\em Proc. Lond. Math. Soc. (3)}, 100(1):155--176, 2010.

\bibitem[Hof07]{Hofer-2007}
R.~Hofer.
\newblock Note on the joint distribution of the weighted sum-of-digits function
  modulo one in case of pairwise coprime bases.
\newblock {\em Unif. Distrib. Theory}, 2(2):35--47, 2007.

\bibitem[IK01]{IndlekoferKatai-2002}
K.-H. Indlekofer and I.~K{\'a}tai.
\newblock On {$q$}-multiplicative functions taking a fixed value on the set of
  primes.
\newblock {\em Period. Math. Hungar.}, 42(1-2):45--50, 2001.

\bibitem[IKL02]{IndlekoferKataiLee-2002}
K.-H. Indlekofer, I.~K{\'a}tai, and Y.-W. Lee.
\newblock On {$q$}-multiplicative functions.
\newblock {\em Publ. Math. Debrecen}, 61(3-4):393--402, 2002.

\bibitem[ILW05]{IndlekoferLeeWagner-2005}
K.-H. Indlekofer, Y.-W. Lee, and R.~Wagner.
\newblock Mean behaviour of uniformly summable {$q$}-multiplicative functions.
\newblock {\em Ann. Univ. Sci. Budapest. Sect. Comput.}, 25:171--194, 2005.

\bibitem[K{\'a}t02]{Katai-2002}
I.~K{\'a}tai.
\newblock On {$q$}-additive and {$q$}-multiplicative functions.
\newblock In {\em Number theory and discrete mathematics ({C}handigarh, 2000)},
  Trends Math., pages 61--76. Birkh\"auser, Basel, 2002.

\bibitem[K{\'a}t09]{Katai-2009}
I.~K{\'a}tai.
\newblock On {$q$}-additive and {$q$}-multiplicative functions.
\newblock In {\em Number theory and applications}, pages 105--126. Hindustan
  Book Agency, New Delhi, 2009.

\bibitem[Kon17]{Konieczny-2017+}
J.~Konieczny.
\newblock {Gowers norms for the Thue-Morse and Rudin-Shapiro sequences}.
\newblock 2017.
\newblock Preprint. {\href{https://arxiv.org/abs/1611.09985}{arXiv:1611.09985
  [math.NT]}}.

\bibitem[Kon18]{Konieczny-2018-semimulti}
J.~Konieczny.
\newblock {M{\"o}bius orthogonality for $q$-semimultiplicative sequences}.
\newblock 2018.
\newblock Preprint. {\href{https://arxiv.org/abs/1808.06196}{arXiv:1808.06196
  [math.NT]}}.

\bibitem[KW17]{KropfWagner-2017}
S.~Kropf and S.~Wagner.
\newblock On {$q$}-quasiadditive and {$q$}-quasimultiplicative functions.
\newblock {\em Electron. J. Combin.}, 24(1):Paper 1.60, 22, 2017.

\bibitem[LM96]{LesigneMauduit-1996}
E.~Lesigne and C.~Mauduit.
\newblock Propri\'et\'es ergodiques des suites {$q$}-multiplicatives.
\newblock {\em Compositio Math.}, 100(2):131--169, 1996.

\bibitem[LMM94]{LesigneMauduitMosse-1994}
E.~Lesigne, C.~Mauduit, and B.~Moss\'e.
\newblock Le th\'eor\`eme ergodique le long d'une suite {$q$}-multiplicative.
\newblock {\em Compositio Math.}, 93(1):49--79, 1994.

\bibitem[Mau05]{Mauclaire-2005}
J.-L. Mauclaire.
\newblock A characterization of some {$q$}-multiplicative functions.
\newblock {\em Acta Arith.}, 120(4):313--336, 2005.

\bibitem[MF67]{MendesFrance-1967}
M.~Mend\`es~France.
\newblock Nombres normaux. {A}pplications aux fonctions pseudo-al\'eatoires.
\newblock {\em J. Analyse Math.}, 20:1--56, 1967.

\bibitem[MPS05]{MauduitPomeranceSarkozy-2005}
C.~Mauduit, C.~Pomerance, and A.~S\'ark\"ozy.
\newblock On the distribution in residue classes of integers with a fixed sum
  of digits.
\newblock {\em Ramanujan J.}, 9(1-2):45--62, 2005.

\bibitem[MR05]{MauduitRivat-2005}
C.~Mauduit and J.~Rivat.
\newblock Propri\'et\'es {$q$}-multiplicatives de la suite {$\lfloor
  n^c\rfloor$}, {$c>1$}.
\newblock {\em Acta Arith.}, 118(2):187--203, 2005.

\bibitem[MRS17]{MRS}
C.~Mauduit, J.~Rivat, and A.~S\'ark\"ozy.
\newblock On the digits of sumsets.
\newblock {\em Canad. J. Math.}, 69(3):595--612, 2017.

\bibitem[MS96]{MauduitSarkozy-1996}
C.~Mauduit and A.~S\'ark\"ozy.
\newblock On the arithmetic structure of sets characterized by sum of digits
  properties.
\newblock {\em J. Number Theory}, 61(1):25--38, 1996.

\bibitem[MS97]{MauduitSarkozy-1997}
C.~Mauduit and A.~S\'ark\"ozy.
\newblock On the arithmetic structure of the integers whose sum of digits is
  fixed.
\newblock {\em Acta Arith.}, 81(2):145--173, 1997.

\bibitem[MS17]{MullnerSpiegelhofer-2017}
C.~M{\"u}llner and L.~Spiegelhofer.
\newblock Normality of the {T}hue-{M}orse sequence along {P}iatetski-{S}hapiro
  sequences, {II}.
\newblock {\em Israel J. Math.}, 220(2):691--738, 2017.

\bibitem[Shi17]{Shi}
R.~Shi.
\newblock {Construction of some Chowla sequences}.
\newblock 2017.
\newblock Preprint.

\bibitem[Spi16]{Spiegelhofer-2016}
L.~Spiegelhofer.
\newblock {Pseudorandomness of the Ostrowski sum-of-digits function}.
\newblock 2016.
\newblock Preprint. {\href{https://arxiv.org/abs/1611.03043}{arXiv:1611.03043
  [math.NT]}}.

\bibitem[Tao12]{Tao-book}
T.~Tao.
\newblock {\em Higher order {F}ourier analysis}, volume 142 of {\em Graduate
  Studies in Mathematics}.
\newblock American Mathematical Society, Providence, RI, 2012.

\bibitem[TZ10]{TaoZiegler-2010}
T.~Tao and T.~Ziegler.
\newblock The inverse conjecture for the {G}owers norm over finite fields via
  the correspondence principle.
\newblock {\em Anal. PDE}, 3(1):1--20, 2010.

\bibitem[TZ12]{TaoZiegler-2012}
T.~Tao and T.~Ziegler.
\newblock The inverse conjecture for the {G}owers norm over finite fields in
  low characteristic.
\newblock {\em Ann. Comb.}, 16(1):121--188, 2012.

\end{thebibliography}

\end{document}